\documentclass[a4paper]{amsart}
\usepackage{amsfonts}
\usepackage{amsmath}
\usepackage{mathrsfs}
\usepackage{graphicx}
\usepackage{amsmath,amsthm,amssymb,amscd}
\usepackage{color}
\usepackage{array}
\usepackage{float}
\usepackage{enumerate}
\usepackage[T1]{fontenc}
\usepackage{pgfplots}
\usetikzlibrary{arrows}
\newtheorem{thm}{Theorem}[section]
\newtheorem{lem}[thm]{Lemma}
\newtheorem{prop}[thm]{Proposition}
\newtheorem{defn}[thm] {Definition}
\newtheorem{exmp}[thm]{Example}
\newtheorem{cor}[thm]{Corollary}
\theoremstyle{remark}
\newtheorem{rem} [thm]{Remark}

\theoremstyle{definition}

\newtheorem{que}[thm]{Question}

\newtheorem{fact}[thm]{Fact}

\DeclareMathOperator{\diam}{diam}
\DeclareMathOperator{\dist}{dist}
\DeclareMathOperator{\Per}{Per}
\DeclareMathOperator{\Int}{Int}
\DeclareMathOperator{\Rec}{Rec}
\DeclareMathOperator{\SA}{SA}

\DeclareMathOperator{\Orb}{Orb}

\newcommand{\hn}{\hat{n}}

\newcommand{\hz}{\hat{z}}
\newcommand{\tx}{\tilde{x}}
\newcommand{\tz}{\tilde{z}}

\newcommand{\hx}{\hat{x}}

\newcommand{\eps}{\varepsilon}

\makeatletter
\@namedef{subjclassname@2020}{%
  \textup{2020} Mathematics Subject Classification}
\makeatother

\definecolor{uuuuuu}{rgb}{0.26666666666666666,0.26666666666666666,0.26666666666666666}
\definecolor{ffqqqq}{rgb}{1,0,0}
\begin{document}

\title[On the structure of $\alpha$-limit sets]{On the structure of $\alpha$-limit sets of backward trajectories for graph maps}

\author[M. Fory\'s]{Magdalena Fory\'s-Krawiec}
\address[M. Fory\'s]{AGH University of Science and Technology, Faculty of Applied
	Mathematics, al. Mickiewicza 30, 30-059 Krak\'ow, Poland}
\email{maforys@agh.edu.pl}
\author[J. Hant\'{a}kov\'{a}]{Jana Hant\'{a}kov\'{a}}
\address[J. Hant\'{a}kov\'{a}]{AGH University of Science and Technology, Faculty of Applied
	Mathematics, al.
	Mickiewicza 30, 30-059 Krak\'ow, Poland
		-- and --
		Mathematical Institute of the Silesian University in Opava, Na Rybn\'i\v {c}ku 1, 74601, Opava, Czech Republic}
\email{jana.hantakova@math.slu.cz}
\author[P. Oprocha]{Piotr Oprocha}
\address[P. Oprocha]{AGH University of Science and Technology, Faculty of Applied
	Mathematics, al.
	Mickiewicza 30, 30-059 Krak\'ow, Poland
	-- and --
	Centre of Excellence IT4Innovations - Institute for Research and Applications of Fuzzy Modeling, University of Ostrava, 30. dubna 22, 701 03 Ostrava 1, Czech Republic.}
\email{oprocha@agh.edu.pl}

\begin{abstract}
In the paper we study what sets can be obtained as $\alpha$-limit sets of backward trajectories in graph maps. We show that in the case of mixing maps,
all those $\alpha$-limit sets are $\omega$-limit sets and for all but finitely many points $x$, we can obtain every $\omega$-limits set as the $\alpha$-limit
set of a backward trajectory starting in $x$. For zero entropy maps, every $\alpha$-limit set of a backward trajectory is a minimal set. In the case of maps with positive entropy, we obtain a partial characterization which is very close to complete picture
of the possible situations.
\end{abstract}

\maketitle

\section{introduction and main results}
 Let a \emph{dynamical system} be defined as a pair $(X,f)$ where $X$ is a compact metric space and $f$ is a continuous map acting on $X$. To understand the dynamical properties of a system it is necessary to analyze the behavior of the trajectories of any point $x\in X$ under the iteration of $f$. Limit sets of trajectories are a helpful tool for 
 understanding of qualitative properties of dynamics.
The \emph{$\omega$-limit set} (set of limit points of forward trajectory of a point $x$; denoted $\omega(x)$), is among fundamental objects in theory of dynamical systems. The first question that comes to mind, is whether a given closed invariant subset of $X$ is the $\omega$-limit set of some point $x\in X$. Finding the answer is hard in general, however some cases are known. For example, a characterization of $\omega$-limit sets of a continuous map acting on the compact interval was provided by Blokh et al. in \cite{BBHS}.  A closely related question asks which dynamical systems may occur as $\omega$-limit sets in larger systems. These abstract $\omega$-limit sets were studied by Bowen \cite{Bow} and Dowker and Frielander \cite{DF}. Of particular interest are invariant sets obtained as limits (in the Hausdorff metric) of $\omega$-limit sets. Sharkovsky proved in \cite{Sh} that every $\omega$-limit sets of continuous map on the interval is contained in the maximal one, and later Blokh et al. in \cite{BBHS} showed that the family of all $\omega$-limit sets of $f$ endowed with the Hausdorff metric is compact.
While the aforementioned results about $\omega$-limit sets were first obtained for interval maps, some of them hold for maps acting on graphs, dendrites, Cantor space and others, e.g. see \cite{ChGSS,KKM,BGKR}.
In general compact metric spaces only partial results are known and are usually hard to obtain (e.g. see \cite{BGOR}).

The properties of $\omega$-limit sets for \emph{graph maps} are to some extent similar to the interval case (however the proofs are usually much harder). Every $\omega$-limit set is contained in a maximal one since the family of $\omega$-limit sets of a graph map is closed with respect to the Hausdorff metric by result of Mai and Shao \cite{MaiShao}. By Blokh's Decomposition Theorem \cite{B1}, there are only four types of maximal $\omega$-limit sets: basic sets, solenoidal sets, circumferential sets and periodic orbits. For a graph map $f$, the topological entropy of $f$ is positive if and only if it possesses a basic sets (i.e. infinite maximal $\omega$-limit sets containing a periodic point; see Hric and M\' alek \cite{M}). The topological characterization of $\omega$-limit sets of graph maps \cite{M} shows that an $\omega$-limit set is a finite set, or an infinite compact nowhere dense set, or a cycle of connected subgraphs. Conversely, whenever a set $A$ is of one of the above forms then there is a graph map $f$ such that $A$ is an $\omega$-limit set for $f$.

As a dual concept to $\omega$-limit sets the \emph{$\alpha$-limit sets} (denoted $\alpha(x)$) were introduced. Intuitively, they represent a ``source'' of the trajectory of a point. While for invertible maps $\alpha$-limit sets can be defined as $\omega$-limits sets of dynamical system with reversed time, for noninvertible maps there are a few possibilities how to construct the limit along the backward trajectory which can not be uniquely defined. One possibility is to take as an $\alpha$-limit set the set of all accumulation points of the set of pre-images $f^{-n}(x)$. This approach was used by Coven and Nitecki \cite{CN}, who showed that for an interval map, a point $x$ is non-wandering if and only if $x\in\alpha(x)$. This approach attracted some attention, e.g. see  Cui and Ding \cite{CD} on $\alpha$-limit sets of unimodal interval maps. Another approach (see \cite{BDLO}), which is studied in the present paper, instead of complete preimages considers a fixed backward branch and its accumulation points forming an \emph{$\alpha$-limit set of a backward branch}.  By results of \cite{BDLO} for interval maps, every $\alpha$-limit set of a backward branch is an $\omega$-limit set while the converse is not true. 
The third approach to $\alpha$-limit sets, proposed by Hero in \cite{Hero}, falls somewhere between two possibilities mentioned above. It considers the union of $\alpha$-limit sets over all backward branches starting at a point $x$, and call obtained set the \emph{special $\alpha$-limit set} (denoted $s\alpha(x)$). Recent studies by Kolyada et al.~\cite{KMS} and Hant\' akov\' a and Roth \cite{HR} provided basic properties of special $\alpha$-limit sets for interval maps. For instance, $s\alpha(x)$ does not need to be closed and its isolated points are always periodic, which is in some contrast to the properties of $\omega(x)$. Outside the realm of one-dimensional dynamics the situation is even more complicated. It has been shown that $s\alpha$-limit sets are always analytic, but not necessarily Borel~\cite{JMR}. If we denote by $\SA(f)$ (respectively, $\omega(f)$) the union of $\alpha$-limit sets of all backward branches (respectively, all $\omega$-limit sets) of a map $f$ and by $\Rec(f)$ the set of all recurrent points of $f$,  then $\Rec(f)\subseteq \SA(f)\subseteq \overline{\Rec(f)}\subseteq \omega(f)$, for every map $f$ on the topological graph (see \cite{SXL}, cf. \cite{Hero}, \cite{BDLO}). It was shown in \cite{STSXQ} that  $\Rec(f)\subseteq \SA(f)$ holds in the special case of maps acting on dendrites with countable set of endpoints and that there are dendrite maps with $\SA(f)\not\subset  \overline{\Rec(f)}$. We show that $\Rec(f)\subseteq \SA(f)$ holds for general dynamical systems in Corollary \ref{trans}.

Our research is motivated by the following question:
\begin{que}\label{que:main}
Let $A=\alpha(\{x_j\}_{j\leq 0})$ be an $\alpha$-limit set of a backward branch $\{x_j\}_{j\leq 0}$ of a map $f$ on topological graph. Is $A$ an $\omega$-limit set? How many different sets $A$
can be generated using backward branches starting at $x_0$?
\end{que}
In the paper we provide full characterization under some additional conditions on $f$, and almost complete picture in general case.

Complete answer to Question~\ref{que:main} in the case of topologically mixing $f\colon G \to G$ on topological graph $G$ is provided in Section \ref{sec:mixing}.
Strictly speaking we prove the following:
\begin{enumerate}
	\item for every $\omega$-limit set $\omega(y)$ in $G$ and every accessible point $x$ in $G$, there is a backward branch starting at $x$ with the $\alpha$-limit set being equal to $\omega(y)$,
	\item every $\alpha$-limit set of any backward branch in $G$ is an $\omega$-limit set of some point in $G$.
\end{enumerate}
Quite different, still complete, picture is obtained for maps $f\colon G \to G$ with zero topological entropy.
For these maps we prove in Section \ref{sec:zeroEnt} that
\begin{enumerate}
	\item family of $\alpha$-limit sets of backward
	branches coincides with the family of minimal sets,
	\item the collection of all $\alpha$-limit sets of backward branches starting at $x$ is rather thin - it contains at most one infinite set.
\end{enumerate}
We also provide an example that in the above case, beyond one infinite minimal set, $\alpha$-limit sets of backward branches starting at $x$ can form quite large family of periodic orbits.

When considering maps with positive entropy, some uncertainty enters our description. In this case, we may observe phenomena specific both for zero entropy maps and for mixing maps,
however tools we use (in Section \ref{sec:PosEnt}) do not allow us to completely reveal the structure of some $\alpha$-limit sets.
We prove that for all but at most countably many points $x$ from a basic set $D$ and every infinite $\omega$-limit set $\omega(y)\subset D$, there exists a backward branch $\{z_j\}_{j\leq 0}$ starting at $x$ such that $\alpha(\{z_j\}_{j\leq 0})=\omega(y)\cup R$ where $R$ is at most countable subset of isolated points of $\alpha(\{z_j\}_{j\leq 0})$. This shows that for a typical point $x$ from a basic set the collection of all $\alpha$-limit sets of backward branches starting at $x$ is abundant. 
By results mentioned earlier, graph maps with positive entropy must contain a basic set, but it may contain also other maximal $\omega$-limit sets which are not limited to zero entropy maps only, that is solenoidal sets, circumferential sets and
periodic orbits. Thus we may detect this kind of $\alpha$-limit sets of backward branches starting at a point $x$ in maps with positive entropy.

As was stated above, our tools do not allow us to answer whether at most countable set $R$ is empty or not, however its possible existence is a result of incomplete control of backward trajectory in the construction rather than a fact. In practice it may happen that these $\alpha$-limit sets behave exactly the same as for other backward trajectories, that is they always coincide with $\omega$-limit sets and all $\omega$-limit sets in basic sets can appear as $\alpha$-limit sets (recall results of \cite{BDLO} that some $\omega$-limit sets are never $\alpha$-limit sets for zero entropy maps).
These aspects of Question~\ref{que:main} remain as open problem for further research.

\section{preliminaries}
Throughout the paper, a \emph{(topological) graph} is a non-degenerate compact connected metric space $G$ containing a finite subset $\mathcal{V}$ such that each connected component of $G\setminus \mathcal{V}$ is homeomorphic to an open interval. A \emph{branching point} is a point in $G$ having no neighborhood homeomorphic to an interval (of any kind). The set of branching points is included in $\mathcal{V}$ hence finite. 
An \emph{arc} is a subset of $G$ homeomorphic to an interval. If $J\subset G$ is an arc with endpoints $x$ and $y$, then it is convenient to write $J=[x,y]$, which means that
	we identify $J$ with interval $[0,1]$ by a homeomorphism $\pi \colon J\to [0,1]$ with $\pi(x)=0, \pi(y)=1$. This way, we may use standard ordering on $[0,1]$ in $J$. In particular, for
	for $b \in [x,y]\setminus\{x,y\}$ we may write $x<b<y$ (using ordering of $J$) and also $[b,y]\subset [x,y]$ is defined in a natural way. For $n \in \mathbb{N}$ a subgraph $St(x) \subset G$ is \emph{$n$-star with center } $x \in G$ if there is a continuous injection $\varphi:St(x)\to \mathbb{C}$ such that $\varphi(x)=0$ and $\varphi(S) = \{re^{\frac{2k\pi i}{n}}: r \in [0,1], k=1,\dots,n\}$.
The \emph{degree} of a point $x \in G$ is given by the following formula:
$$
\deg(x) = \max\{ n \in \mathbb{N}: \text{ there exists an }n\text{-star in }G\text{ with center } x \} \in [1,\infty).
$$
If $\deg(x)=1$, then $x$ is called an \emph{endpoint}, while the degree of branching points is always at least $3$. We denote the set of all endpoints in $G$ by $End(G)$ and the set of all branching points in $G$ by $Br(G)$.
By a \emph{(graph) map} we mean a dynamical system on a graph, that is, a continuous map $f\colon G\rightarrow G$. The \emph{orbit of a point} $x \in G$ is the set $\Orb_f(x) = \{f^n(x):n \geq 0\}$, while the \emph{orbit of a set} $A \subset G$ is the set $\Orb_f(A) = \cup_{n=0}^{\infty} f^n(A)$. If the function $f$ in the above definitions is clear from the context, we use the notation $\Orb(x)$ and $\Orb(A)$. By $\Per(f)$ we denote the set of \textit{periodic points} of $f$, that is points with the property that $f^p(x)=x$ for some $p>0$. The smallest such $p$  is the \emph{period} of a point $x \in \Per(f)$. A set $A$ is \emph{invariant} if $f(A)\subseteq A$ and it is \emph{strongly invariant} if $f(A)=A$.

Map $f\colon G\to G$ is \emph{transitive} if for every pair of nonempty open subsets $U,V\subset G$ there is some integer $n>0$ such that $f^n(U)\cap V\neq \emptyset$ and \emph{totally transitive} if $f^n$ is transitive for all $n >0$. Map $f$ is \emph{sensitive} if there is $\delta>0$ such that for every nonempty open $U\subset G$ there is $n>0$ such that $\diam f^n(U)>\delta$ and it is \emph{mixing} if for every pair of nonempty open subsets $U,V\subset G$ there is an $N>0$ such that $f^n(U)\cap V\neq \emptyset$ for $n\geq N$. A point $x \in G$ is \emph{non-wandering} if for every neighborhood $U$ of $x$ and every $N>0$ there is some $n>N$ such that $f^n(U)\cap U \neq \emptyset$. If the opposite holds then we say $x$ is a \emph{wandering} point.

For a mixing graph map $f\colon G\to G$ we  define the set $\mathcal{I}(f)$ of \emph{inaccesible points} of $f$ as follows:
	\begin{equation}\label{eq:inac}
		\mathcal{I}(f) = G\setminus\bigcap_{U \in \mathcal{G}}\bigcup_{k=0}^{\infty}\Int f^k(U),
	\end{equation}
	where $\mathcal{G}$ is the family of all subgraphs of $G$. By the results of \cite{B1} we have that $\mathcal{I}(f)$ is a finite strongly invariant set and hence it consists of periodic points. We say that $x \in G$ is an \emph{accessible point} if $x \in G\setminus \mathcal{I}(f)$. 

A point $y$ belongs to the \emph{$\omega$-limit set of a point $x$}, denoted by $\omega_f(x)$, if and only if there is a strictly increasing sequence of natural numbers $\{n_i\}_{i\geq 0}$ such that $f^{n_i}(x)\to y$ as $i\to \infty$. We denote $\omega(f) = \bigcup_{x \in G}\omega_f(x)$. We say that $x$ is \emph{recurrent} if $x \in \omega_f(x)$ and by $\Rec(f)$ we denote the set of recurrent points for map $f$.   
A \emph{backward branch} of a point $x\in G$ is any sequence $\{x_i\}_{i\leq 0}\subset G$ such that $x_0=x$
and $f(x_i)=x_{i+1}$ for each $i<0$. 
A point $y$ belongs to the \emph{$\alpha$-limit set of a backward branch} $\{x_i\}_{i\leq 0}$, denoted by $\alpha_f(\{x_i\}_{i\leq 0})$, if and only if there is a strictly decreasing sequence of negative integers $\{n_i\}_{i\geq 0}$ such that $x_{n_i}\to y$ as $i\to \infty$. It is easy to see that both $\omega$-limit sets and $\alpha$-limit sets of backward branches are closed strongly invariant sets. We denote by $\SA(f)$ the union of all $\alpha$-limit sets of backward branches in $G$. If the function from the definition of $\omega$-limit set or $\alpha$-limit set of a backward branch is clear, we use the notation $\omega(x)$, $\alpha(\{x_j\}_{j\leq 0})$.

The following result holds for all transitive dynamical systems.
\begin{prop}
Let $(X,f)$ be a transitive dynamical system. For every $x\in X$ there is a backward branch $\{x_i\}_{i\leq 0}$ such that $x\in \alpha_f(\{x_i\}_{i\leq 0})$.
\end{prop}
\begin{proof}
Let $U_1^0=\bar{B}(x,1)$. Then there is $n_1>0$ such that $f^{n_1}\bar{B}(x,1/2)\cap U_1^0\neq \emptyset$ and the set $U_2^{-n_1}=\bar{B}(x,1/2)\cap f^{-n_1}(U_1^0)$ is non-empty. For $i=0,1,\ldots, n_1$, denote $U_2^{-n_1+i}=f^i(U_2^{-n_1})$. Clearly $U_2^0\subset U_1^0$. Next, assume that a closed set $U_{k}^{-n_1-n_2-\ldots-n_k}$ with nonempty interior is defined such that $U_{k}^{-n_1-n_2-\ldots-n_k}\subset B(x,1/k)$. For simplicity we will denote  $\sigma_i=\sum_{j=1}^{i}n_j$, for every $i>0$ such that $n_1,n_2,\ldots, n_i$ is defined. There is $n_{k+1}>0$ such that $f^{n_{k+1}}(\bar{B}(x,1/(k+1))\cap U_{k}^{-\sigma_k})\neq \emptyset$. Put $U_{k+1}^{-\sigma_k-n_{k+1}}=\bar{B}(x,1/(k+1)\cap f^{-n_{k+1}}(U_{k}^{-\sigma_k})$. Then as before denote $U_{k+1}^{-\sigma_{k+1}+i}=f^i(U_{k+1}^{-\sigma_{k+1}})$, for $i=0,1,\ldots, \sigma_{k+1}$ .
  Clearly $U_{k+1}^{-\sigma_j}\subset U_{s}^{-\sigma_j}$ for any $j\leq k+1$ and any  $j\leq s\leq k$.
Replacing $U_{k+1}^{-\sigma_{k+1}}$ with a smaller closed ball contained in it, we may additionally requite
that $\diam U_{k+1}^{-i}<1/(k+1)$ for each $i\leq \sigma_{k+1}$.

Then for every $k\geq 1$ there is a unique point:
$$
x_{-\sigma_k}\in \bigcap_{j=k}^\infty U_{j}^{-\sigma_k}.
$$
Furthermore, by the construction $f^{n_{k+1}}(x_{-\sigma_{k+1}})=x_{-\sigma_k}$
and $d(x_{-\sigma_k},x)<1/k$. Putting $x_0=f^{n_1}(x_{-n_1})$ and $x_{-\sigma_{k}+i}=f^i(x_{ -\sigma_{k}})$, for every $i=1,\ldots, n_{k}-1$ and $k\geq 1$,
we define the desired backward branch $\{x_i\}_{i\leq 0}$.
\end{proof}
For every recurrent point $x\in X$, the dynamical system $(\omega_f(x),f)$ is transitive and obviously $x\in \omega_f(x)$. Thus we have the following corollary.
\begin{cor}\label{trans}
Let $(X,f)$ be a dynamical system. If $x\in \Rec(f)$ then there is a backward branch $\{x_i\}_{i\leq 0}$ such that $x\in \alpha_f(\{x_i\}_{i\leq 0})$. In particular, $\Rec(f)\subset \SA(f)$.
\end{cor}

For fixed $n\geq 0$ and $\eps>0$ we define the \emph{Bowen ball}  as follows:
$$
B_n(x,\eps) = \{y \in G: d(f^i(x),f^i(y))\leq \eps \text{ for } i=0,\dots,n  \}
$$
and by $B'_n(x,\eps)$ we denote the connected component of the Bowen ball $B_n(x,\eps)$ that contains $x$. 

For any nonempty compact subsets $U,V \subset G$ we define their \emph{Hausdorff distance} by:
$$
d_H(U,V) = \max\{ \dist(u,V), \dist(v,U): u\in U, v \in V  \},
$$
where $\dist(x,U) =\inf\{ d(x,y): y \in U \}$.
We denote by $K(G)$ be the set of all compact subsets of $G$ equipped with the \textit{Hausdorff metric} $d_H$. We call a set $A\subseteq G$ the \emph{Hausdorff limit} of the sequence of compact sets $\{A_i\}_{i\geq 0}$ if $\{A_i\}_{i\geq 0}$ converges to $A$ in the metric $d_H$.

A subset $Y\subset G$ is \emph{internally chain transitive} if for every pair of points $u,v \in Y$ and every $\eps>0$ there is a finite sequence $z_0,\dots, z_n$ of points in $Y$ such that $z_0=u, z_n=v$ and $d(f(z_i),z_{i+1})<\eps$ for $i=0,\dots,n-1$. It is well-known fact that in any dynamical system $(X,f)$, every $\omega$-limit set is internally chain transitive and the same holds for $\alpha$-limit sets of a backward branch by \cite[Lemma 2.1]{HSZ}. Therefore $\alpha$-limit sets of a backward branch share the following property of $\omega$-limit sets.

\begin{lem}\label{lem:alphaPer}
Let $(X,f)$ be a dynamical system and $\alpha(\{x_j\}_{j\leq 0})\subset X$ be an $\alpha$-limit set of a backward branch $\{x_j\}_{j\leq 0}$. Then every periodic orbit that lies in $\alpha(\{x_j\}_{j\leq 0})$ but does not coincide with $\alpha(\{x_j\}_{j\leq 0})$ is not isolated in $\alpha(\{x_j\}_{j\leq 0})$. In particular, if $\alpha(\{x_j\}_{j\leq 0})$ consists of finitely many points, then these points form a single periodic orbit.
\end{lem}
\begin{proof}
Assume that the statement does not hold, i.e. there is a periodic orbit $\Orb(p)\subset \alpha(\{x_j\}_{j\leq 0})$ isolated in $\alpha(\{x_j\}_{j\leq 0})$ and a point $q\in \alpha(\{x_j\}_{j\leq 0})\setminus\Orb(p)$. Let $\eps>0$ be such that $\cup_{x\in\Orb(p)}B(x,\eps)$ does not contain any point from $\alpha(\{x_j\}_{j\leq 0})\setminus\Orb(p)$. As $\alpha(\{x_j\}_{j\leq 0})$ is internally chain transitive, there is a chain $z_0,\dots,z_n$ of points from $\alpha(\{x_j\}_{j\leq 0})$ joining $p$ and $q$ with $d(f(z_i),z_{i+1})<\eps$ for $i=0,\dots,n-1$. We have $z_i\in \Orb(p)$, for every $i=0,\dots,n-1$, by the choice of $\eps$. But then $d(f(z_{n-1}),q)<\eps$ implies $q\in\Orb(p)$ which contradicts the assumption. Therefore $ \alpha(\{x_j\}_{j\leq 0})\subseteq \Orb(p)$ and since $ \alpha(\{x_j\}_{j\leq 0})$ is an invariant set, $ \alpha(\{x_j\}_{j\leq 0})= \Orb(p)$.
\end{proof}
\begin{rem}\label{rem:omegaFinite}
Since every $\omega$-limit set $\omega(y)\subset G$ is internally chain transitive, we obtain by the same reasoning as above that if $\omega(y)$ is finite then it is a single periodic orbit.
\end{rem}

\section{Mixing graph maps}\label{sec:mixing}
We will show that any mixing graph map $f\colon G\to G$ has the following properties:
\begin{enumerate}
	\item for every $\omega$-limit set $\omega(y)$ in $G$ and every accessible point $x$ in $G\setminus \mathcal{I}(f)$, there is a backward branch starting at $x$ with the $\alpha$-limit set being equal to $\omega(y)$,
	\item every $\alpha$-limit set of any backward branch in $G$ is an $\omega$-limit set of some point in $G$.
\end{enumerate}

We start with the construction of the backward branch starting at any accessible point whose $\alpha$-limit set equals $\omega(y)$ for a chosen point $y \in G$. Lemma \ref{lem:MixPer} and Lemma \ref{lem:MixInfinite} distinguish two cases depending on the cardinality of $\omega(y)$. In both cases, for infinite and finite $\omega(y)$, the idea of construction is based on the properties of Bowen balls expressed in the following two lemmas from \cite{HKO}.
\begin{lem}\cite[Lemma 10.4]{HKO}\label{lem:10.4}
 Let $f\colon G\to G$ be a mixing graph map. If $0<\eps<\frac{1}{2}\diam G$ and $\delta>0$ then there is an $N=N(\eps,\delta)>0$ such that $B'_n(x,\eps)\subset B(x,\delta)$ for all $x\in G$ and all $n\geq N$.
\end{lem}
\begin{lem}\cite[Lemma 10.5]{HKO}\label{lem:10.5}
 Let $f\colon G\to G$ be a mixing graph map. For every $\eps>0$ there is a constant $\eta=\eta(\eps)$ such that
$$0<\eta\leq \diam f^n(B'_n(x,\eps))$$
for every $n\geq 0$ and $x\in G$.
\end{lem}

\begin{lem}\label{lem:MixPer}
 Let $f\colon G\to G$ be a mixing graph map and $y \in G$ such that $\omega(y)$ is finite. There exists an open connected set $U \subset G$ such that, for every $x_0 \in U$ there is a backward branch $\{\tx_j\}_{j\leq 0}$ with $\alpha(\{\tx_j\}_{j\leq 0})=\omega(y)$.
\end{lem}
\begin{proof}
	By Remark \ref{rem:omegaFinite}, every finite $\omega(y)$ is an orbit of some periodic point $p\in G$. First assume that $p \in G$ is a fixed point, i.e. $f^i(p)=p$ for all $i\geq 0$, and denote $\deg(p)=L$. Choose $\eps>0$ and let $\eta>0$ be the constant from Lemma \ref{lem:10.5} such that for every $n \in \mathbb{N}$ and $x \in G$ we have: 
	$$0<\eta\leq\diam f^n(B'_n(x,\eps)).$$
	Decreasing $\eta$ if necessary, we can choose $\{S^{(1)}_i\}_{i\geq 0},\dots \{ S^{(L)}_i \}_{i\geq 0}$ where each $\{ S^{(l)}_i \}_{i\geq 0}$ is a nested sequence of closed arcs in $G$ such that $\diam S^{(l)}_0 = \frac{\eta}{2}$ and $\diam S^{(l)}_{i+1} = \frac12\diam S^{(l)}_i $ for every $1\leq l\leq L$ and every $i\geq 0$, and the fixed point $p$ is the unique element in the intersection $\bigcap_{l=1}^LS^{(l)}_i$, for every $i\geq 0$.
		
	Choose $x^{(1)}_0,\dots, x^{(L)}_0 \in G$ such that $x^{(l)}_0 \in S^{(l)}_0$ and $d(x^{(l)}_0,p)=\frac{\eta}{4}$ for $1\leq l\leq L$. By Lemma \ref{lem:10.4} there exists $N_0 = N(\eps,\frac{\eta}{8})>0$ such that for $n>N_0$ and every $1\leq l \leq L$ we have: 
	$$B'_n(x^{(l)}_0,\eps)\subset B(x^{(l)}_0,\frac{\eta}{8})\subset S^{(l)}_0.$$
	Pick $x^{(l)}_1 \in S^{(l)}_1$, $1\leq l \leq L$ such that  $d(x^{(l)}_1,p)=\frac{\eta}{2^3}$. Again by Lemma \ref{lem:10.4} we get $N_1>0$ such that for every $n>N_1$ we have: 
	$$B'_n(x^{(l)}_1,\eps)\subset B(x_1^{(l)},\frac{\eta}{2^4})\subset S_1^{(l)}.$$
	By Lemma \ref{lem:10.5} we know that $\eta \leq \diam f^{n}(B'_n(x^{(l)}_1,\eps))$ for $1\leq l \leq L$, which implies $\eta\leq \diam f^n(S_1^{(l)})$ as well for any $1\leq l \leq L$ and $n>N_1$. As $S_0^{(l)}\cap S_1^{(l)}=\{p\}$, we have that 
	for every $1\leq l \leq L$  and any $n>N_1$ there exists some $m \in \{1,\dots, L\}$ such that $f^n(S_1^{(l)})\supset S_0^{(m)}$. For any $l \in \{1,\dots,L\}$ we denote:
	$$
	n_1^{(l)} = \min\{ n\leq N_1+1: \text{ there exists }m \in \{1,\dots,L\} \text{ such that } f^n(S_1^{(l)})\supset S_0^{(m)} \}.
	$$ 
	Note that for $n<n^{(l)}_1$ we have $f^n\left(S^{(l)}_1\right) \subset \bigcup_{m=1}^LS^{(m)}_0$, so in particular the diameter of $f^n\left(S^{(l)}_1\right)$ is less than $\eta$.
	
	Now for every $1\leq l \leq L$ we are going to construct a sequence $\{n_i^{(l)}\}_{i>0}$, whose element $n_i^{(l)}$ indicates the first iteration of $f$ such that $f^{n_i^{(l)}}(S_i^{(l)}) \supset S_{i-1}^{(m)}$ for some $m \in \{1,\dots,L\}$.
	Fix  $l \in \{1,\dots, L\}$. 
	Let $N_{k+1}>0$ be the constant from Lemma \ref{lem:10.4} such that for $n>N_{k+1}$ we have: 
	$$B'_n(x_{k+1}^{(l)},\eps)\subset B(x_{k+1}^{(l)},\frac{\eta}{2^{k+3}})\subset S_{k+1}^{(l)}.$$ 
	 By Lemma \ref{lem:10.5} we have that $\eta \leq \diam f^{n}(B'_n(x^{(l)}_{k+1},\eps))$ for $1\leq l \leq L$, which implies $\eta\leq \diam f^n(S_{k+1}^{(l)})$ as well for any $1\leq l \leq L$ and $n>N_{k+1}$. As $S_{k+1}^{(l)}\cap S_k^{(l)}=\{p\}$, we have that for  $n>N_{k+1}$ each $f^n(S_{k+1}^{(l)})$ covers $S_0^{(m)}$ for some $1\leq m \leq L$. In other words, for each $n>N_{k+1}$ there exists some $m \in \{1,\dots, L\}$ such that $f^n(S_{k+1}^{(l)})\supset S_k^{(m)}$. 
	Denote:
	$$
	n_{k+1}^{(l)} = \min\{ n\leq N_{k+1}+1: \text{ there exists }m \in \{1,\dots,L\} \text{ such that } f^n(S_{k+1}^{(l)})\supset S_k^{(m)} \}.
	$$ 
	Note that for $n<n^{(l)}_{k+1}$ we have $f^n\left(S^{(l)}_{k+1}\right) \subset \bigcup_{m=1}^LS^{(m)}_k$, so in particular the diameter of $f^n\left(S^{(l)}_{k+1}\right)$ is less than $\eta/2^k$.
	
	We are going to construct a nested sequence of closed sets $\{Z_i\}_{i\geq 0}$ such that $Z_i\subset \{1,\ldots, L\}^{\mathbb{N}_0}$ for each $i\geq 0$ as follows. A point $z = \{z_n\}_{n\geq 0}$ belongs to $Z_0$ if the following holds:
	$$
	z_1=l, z_0=m \text{ for }m,l \in \{1,\dots,L\} \text{ such that }S_0^{(m)} \subset f^{n_1^{(l)}}(S_1^{(l)}).
	$$
	Below we use the notation $z_{[0,K]}$ to denote a finite sequence of symbols $z_0z_1\dots z_K$ over the given alphabet $\{1,\ldots,L\}$. For $i>0$ a point $w$ belongs to $Z_i$ if there exists some $z \in Z_{i-1}$ such that $w_{[0,i-1]} = z_{[0,i-1]}$ and the following holds:
	$$
	w_i = l, w_{i-1}=m \text{ for }m,l \in \{1,\dots, L\}\text{ such that }  S_{i-1}^{(m)}\subset f^{n_{i}^{(l)}}(S_{i}^{(l)}).
	$$
	
	The intersection $Z = \bigcap_{i\geq 0}Z_i$ is non-empty, since $\{Z_i\}_{i\geq 0}$ is nested sequence of compact sets, so fix some $z \in Z$. Depending on the first symbol of $z$ put $U = \Int S_0^{(l)}$ provided that $z_0 = l$, $l \in \{1,\dots, L\}$. For every $i\geq 1$ denote :
	$$
	k_i = n_i^{(l)} \text{ for such an }l \in \{1,\dots, L\} \text{ that }z_i = l
	$$
	and
	$$	
	\tilde{S}_i = S^{(l)}_i \text{ for such an }l \in \{1,\dots, L\}\text{ that } z_i=l.
	$$
	
	Then $f^{k_i}(\tilde{S}_i)\supset \tilde{S}_{i-1}$, for every $i>1$.  Let $\tx_0$ be an arbitrary point from $U$. The backward branch $\{ \tx_j \}_{j\leq 0}$ starting at $\tx_0$ is defined as follows. For $j=-\sum_{i=1}^mk_i$ for consecutive $m=1,2,3,\ldots$, we pick
	$$
		\tx_j \in \tilde{S}_m \text{ such that }f^{k_{m}}(\tx_j) = \tx_{j+k_{m}} 
	$$
	and for all other $j$ we define
	$$
		\tx_{j} =f(\tx_{j-1}) \text{ for } -\sum_{s=1}^{m}k_s < j < -\sum_{s=1}^{m-1}k_s, m> 1, \text{ or } -k_1<j<0.
	$$
	The above conditions guarantee that $\tx_j$ is well defined for $j\leq 0$ and $\{\tx_j\}_{j\leq 0}$ is a~backward branch of some point from $U$. 
By the construction for every $m>0$ we will find some $t>0$ such that $\{\tx_j\}_{j\leq -t} \subset B(p,\frac{\eta}{2^{m+1}})$, so altogether $\{p\} =\alpha(\{\tx_j\}_{j\leq 0})$.

Now assume that $p$ is a point of period $K>1$. We use the above result for map $f^K$ for which $p$ is a fixed point to get the backward branch $\{\tx_j\}_{j\leq 0}$ with $\tx_0 \in U$. Hence $\{\tilde{y}_j\}_{j\leq 0}$ defined as follows:
\begin{align*}
	\tilde{y}_{Kj} & = \tx_j, j\leq 0,\\
	\tilde{y}_s &= f(\tilde{y}_{s-1}) \text{ for } (K+1)j <s< Kj, j\leq 0 \\
\end{align*}
is the backward branch of $\tilde{y}_0 = \tx_0 \in U$ for map $f$. By continuity, and the fact that $\{p\} =\alpha_{f^K}(\{\tx_j\}_{j\leq 0})$
we obtain that $\Orb(p)=\alpha_{f}(\{\tilde{y}_j\}_{j\leq 0})$.
\end{proof}
\begin{lem}\label{lem:MixInfinite}
	Let $f\colon G\to G$ be a mixing graph map and $y \in G$ such that $\omega(y)$ is infinite. There exists an open connected set $U \subset G$ such that, for every $\tx_0 \in U$, there is a backward branch $\{\tx_j\}_{j\leq 0}$ with $\alpha(\{\tx_j\}_{j\leq 0})=\omega(y)$.
\end{lem}
\begin{proof}
Let $a\in\omega(y)$ be an accumulation point of the infinite set $\omega(y)$ and $J=[a,b]$ be an arc such that $a$ is an accumulation point of the set $\Lambda = \omega(y)\cap[a,b]$. We may assume that $J\setminus \{a,b\}$ is an open free arc, that means $(J\setminus \{a,b\})\cap Br(G)=\emptyset$.
Let  $\{\eps_n\}_{n\geq0}$ be a sequence of positive numbers such that $\eps_0 = \frac12 \diam J$ and  $\eps_{n+1}<\frac12\eps_n$ for every $n\geq0$. Now let $\{\eta_n\}_{n \in \mathbb{N}}$ be the sequence of constants from Lemma \ref{lem:10.5}, such that for every $n,k \in \mathbb{N}$ and $x \in G$ we have:
$$
0<\eta_n\leq \diam f^k(B'_k(x,\eps_n)).
$$
Choose some $l_0,m_0,r_0 \in \Lambda$ such that $d(l_0,r_0)<\frac{\eta_1}{4}$ and:
$$
a<l_0<m_0<r_0.
$$
As $l_0, r_0 \in \omega(y)$ there exist $x_0,\hx_0 \in \Orb(y)$ such that $d(x_0,l_0)<\frac{\eta_1}{8}$ and $d(\hx_0,r_0)<\frac{\eta_1}{8}$. Define $\gamma_0 = \frac18\min\{d(l_0,x_0), d(r_0,\hx_0),\eta_2\}$ and note that in particular $\gamma_0<\frac{\eta_1}{64}$.  Let $N_0 = N(\eps_0,\gamma_0)>0$  be the constant from Lemma \ref{lem:10.4} and take $n_0,\hn_0~>~N_0$ such that $B'_{n_0}(x_0,\eps_0)\subset B(x_0,\gamma_0)$ and $B'_{\hn_0}(\hx_0,\eps_0)\subset B(\hx_0,\gamma_0)$.  

Choose $l_1, m_1, r_1 \in \Lambda$ such that $\d(l_1,r_1)< \frac{\eta_2}{4}$ and:
$$
a<l_1<m_1<r_1<l_0.
$$
Again there exist $x_1,\hx_1 \in \Orb(y)$ with $d(l_1,x_1)<\gamma_0$, and $d(\hx_1,r_1)<\gamma_0$. Denote $\gamma_1 = \frac18\min\{ d(l_1,x_1), d(r_1,\hx_1),\eta_3  \}$ and let $N_1 = N(\eps_1,\gamma_1)>0$ be the constant from Lemma \ref{lem:10.4}. Take $n_1,\hn_1>N_1$ such that:
\begin{align*}
	B'_{n_1}(x_1,\eps_1)&\subset B(x_1,\gamma_1) \text{ and } d(f^{n_1}(x_1),m_0)<\gamma_1 \text{ and }\bigcup_{i=0}^{n_1-1}B(f^i(x_1),\eps_1)\supset \omega(y), \\
	B'_{\hn_1}(\hx_1,\eps_1)&\subset B(\hx_1,\gamma_1) \text{ and } d(f^{\hn_1}(\hx_1),m_0)<\gamma_1 \text{ and }\bigcup_{i=0}^{n_1-1}B(f^i(\hx_1),\eps_1)\supset \omega(y). 
\end{align*}
By  Lemma \ref{lem:10.5} we have:
\begin{align*}
	\eta_1&\leq \diam f^{n_1}(B'_{n_1}(x_1,\eps_1)),\\
	\eta_1&\leq \diam f^{\hn_1}(B'_{\hn_1}(\hx_1,\eps_1)). 
\end{align*}
Note that the following inequalities hold:
\begin{equation}\label{est1}
	\begin{split}
d(f^{n_1}(x_1),m_0)+d(x_0,m_0)+\diam B'_{n_0}(x_0,\eps_0)&<\frac{\eta_1}{64}+\frac{\eta_1}{4}+\frac{\eta_1}{8}<\frac{\eta_1}{2},\\
d(f^{n_1}(x_1),m_0)+d(\hx_0,m_0)+\diam B'_{\hn_0}(\hx_0,\eps_0)&<\frac{\eta_1}{64}+\frac{\eta_1}{4}+\frac{\eta_1}{8}<\frac{\eta_1}{2},
\end{split}
\end{equation}
hence at least one of the following inclusions must hold:
$$
f^{n_1}(B'_{n_1}(x_1,\eps_1))\supset B'_{n_0}(x_0,\eps_0) \text{ or }f^{n_1}(B'_{n_1}(x_1,\eps_1))\supset B'_{\hn_0}(\hx_0,\eps_0).
$$
Analogously, we also have:
\begin{equation}\label{est2}
	\begin{split}
	d(f^{\hn_1}(\hx_1),m_0)+d(x_0,m_0)+\diam B'_{n_0}(x_0,\eps_0)&<\frac{\eta_1}{64}+\frac{\eta_1}{4}+\frac{\eta_1}{8}<\frac{\eta_1}{2},\\
	d(f^{\hn_1}(\hx_1),m_0)+d(\hx_0,m_0)+\diam B'_{\hn_0}(\hx_0,\eps_0)&<\frac{\eta_1}{64}+\frac{\eta_1}{4}+\frac{\eta_1}{8}<\frac{\eta_1}{2}.
	\end{split}
\end{equation}
and therefore, as before:
$$
f^{\hn_1}(B'_{\hn_1}(\hx_1,\eps_1))\supset B'_{n_0}(x_0,\eps_0) \text{ or }f^{\hn_1}(B'_{\hn_1}(\hx_1,\eps_1))\supset B'_{\hn_0}(\hx_0,\eps_0).
$$
Now we are going to construct sequences of Bowen balls $\{B_{n_k}(x_k,\eps_k)\}_{k\geq 0}$ and  $\{B_{\hn_k}(\hx_k,\eps_k)\}_{k\geq 0}$ with the properties as follows for every $k\geq 1$:
	\begin{enumerate}
		\item there are some properly chosen points $l_k,m_k,r_k \in \Lambda$ with $d(l_k,r_k)<\frac{\eta_{k+1}}{4}$ and $a<l_k<m_k<r_k<l_{k-1}$ for which we may find $x_k, \hx_k \in \Orb(y)$ such that $x_k$ is within $\gamma_{k-1}$ distance from $l_k$ and $\hx_k$ is within $\gamma_{k-1}$ distance from $r_k$, where $\gamma_{k-1}=\frac18\min\{d(x_{k-1},l_{k-1}), d(\hx_{k-1},r_{k-1})\}$.
		\item $n_k, \hn_k >N_k$ 
		\item 	$f^{n_i}(B'_{n_i}(x_i,\eps_i))\supset B'_{n_{i-1}}(x_{i-1},\eps_{i-1}) \text{ or } f^{n_i}(B'_{n_i}(x_i,\eps_i))\supset B'_{\hn_{i-1}}(\hx_{i-1},\eps_{i-1})$,
		\item $f^{\hn_i}(B'_{\hn_i}(\hx_i,\eps_i))\supset B'_{n_{i-1}}(x_{i-1},\eps_{i-1}) \text{ or } f^{\hn_i}(B'_{\hn_i}(\hx_i,\eps_i))\supset B'_{\hn_{i-1}}(\hx_{i-1},\eps_{i-1})$
	\end{enumerate}

Assume the above conditions are fulfilled for $i=0,\dots,k$ for some $k\geq 1$.
We proceed with the construction to get $B_{n_{k+1}}(x_{k+1},\eps_{k+1})$ and $B_{\hn_{k+1}}(\hx_{k+1},\eps_{k+1})$. Choose $l_{k+1},m_{k+1},r_{k+1} \in\Lambda$ such that $d(l_{k+1},r_{k+1})<\frac{\eta_{k+2}}{4}$ and:
$$
a<l_{k+1}<m_{k+1}<r_{k+1}<l_k.
$$
Take some $x_{k+1}, \hx_{k+1} \in \Orb(y)$ such that $d(x_{k+1},l_{k+1})<\gamma_k$ and $d(\hx_{k+1},r_{k+1})<\gamma_k$, let $\gamma_{k+1} = \frac18\min\{ d(l_{k+1},x_{k+1}), d(r_{k+1},\hx_{k+1})),\eta_{k+3} \}$. Apply Lemma \ref{lem:10.4} to obtain $N_{k+1} = N(\eps_{k+1},\gamma_{k+1})>0$ and pick $n_{k+1}, \hn_{k+1}>N_{k+1}$ such that:
\begin{align*}
	B'_{n_{k+1}}(x_{k+1},\eps_{k+1})\subset B(x_{k+1},\gamma_{k+1}) &\text{ and }d(f^{n_{k+1}}(x_{k+1}),m_k)<\gamma_{k+1},\\
	B'_{\hn_{k+1}}(\hx_{k+1},\eps_{k+1})\subset B(\hx_{k+1},\gamma_{k+1}) &\text{ and }d(f^{\hn_{k+1}}(\hx_{k+1}),m_k)<\gamma_{k+1},\\
	\bigcup_{i=0}^{n_{k+1}-1}B(f^i(x_{k+1}),\eps_{k+1})&\supset \omega(y),\\
	\bigcup_{i=0}^{\hn_{k+1}-1}B(f^i(\hx_{k+1}),\eps_{k+1})&\supset \omega(y).
\end{align*}
Next, by Lemma \ref{lem:10.5} we have:
\begin{align*}
	\eta_{k+1}&\leq \diam f^{n_{k+1}}(B'_{n_{k+1}}(x_{k+1},\eps_{k+1})),\\
	\eta_{k+1}&\leq \diam f^{\hn_{k+1}}(B'_{\hn_{k+1}}(\hx_{k+1},\eps_{k+1})). 
\end{align*}
The estimations analogous to those in (\ref{est1}) and (\ref{est2}) imply that:
$$
f^{n_{k+1}}(B'_{n_{k+1}}(x_{k+1},\eps_{k+1}))\supset B'_{n_k}(x_k,\eps_k) \text{ or }f^{n_{k+1}}(B'_{n_{k+1}}(x_{k+1},\eps_{k+1}))\supset B'_{\hn_k}(\hx_k,\eps_k)
$$
and:
$$
f^{\hn_{k+1}}(B'_{\hn_{k+1}}(\hx_{k+1},\eps_{k+1}))\supset B'_{\hn_k}(x_k,\eps_k) \text{ or }f^{\hn_{k+1}}(B'_{\hn_{k+1}}(\hx_{k+1},\eps_{k+1}))\supset B'_{\hn_k}(\hx_k,\eps_k).
$$
The induction is completed.

Now we will perform a construction similar to the one from Lemma \ref{lem:MixPer}. We are going to construct a nested sequence of closed sets $\{Z_i\}_{i\geq 0}$ such that $Z_i \subset \{l,r\}^{\mathbb{N}_0}$ for each $i\geq 0$ as follows. A point $z=\{z_n\}_{n\geq 0}$ from $ \{l,r\}^{\mathbb{N}_0}$ is an element of $Z_0$ if one of the following holds:
\begin{align*}
	z_0=l &\text{ and } \left(B'_{n_0}(x_0,\eps_0)\subset f^{n_1}(B'_{n_1}(x_1,\eps_1)) \text{ or }B'_{n_0}(x_0,\eps_0)\subset f^{\hn_1}(B'_{\hn_1}(\hx_1, \eps_1))  \right),\\
	z_0=r &\text{ and } \left( B'_{\hn_0}(\hx_0,\eps_0)\subset f^{n_1}(B'_{n_1}(x_1,\eps_1)) \text{ or } B'_{\hn_0}(\hx_0,\eps_0)\subset f^{\hn_1}(B'_{\hn_1}(\hx_1, \eps_1)) \right).
\end{align*} 
For $i>0$ an infinite sequence a point $w$ belongs to $Z_i$ if there exists some $z \in Z_{i-1}$ such that $z_{[0,i-1]}=w_{[0,i-1]}$ and one of the following holds:
\begin{align*}
	w_i&=l, w_{i-1}=l\text{ and }B'_{n_{i-1}}(x_{i-1},\eps_{i-1})\subset f^{n_{i}}(B'_{n_{i}}(x_{i},\eps_{i})),\\
	w_i&=l, w_{i-1}=r\text{ and }B'_{n_{i-1}}(x_{i-1},\eps_{i-1})\subset f^{\hn_{i}}(B'_{\hn_{i}}(\hx_{i},\eps_{i})),\\
	w_i&=r, w_{i-1}=l\text{ and }B'_{\hn_{i-1}}(\hx_{i-1},\eps_{i-1})\subset f^{n_{i}}(B'_{n_{i}}(x_{i},\eps_{i})),\\
	w_i&=r, w_{i-1}=r\text{ and }B'_{\hn_{i-1}}(\hx_{i-1},\eps_{i-1})\subset f^{n_{i}}(B'_{\hn_{i}}(\hx_{i},\eps_{i})). 
\end{align*}
The intersection $Z = \bigcap_{i\geq 0} Z_i $ is non-empty, since $\{Z_i\}_{i\geq 0}$ is nested sequence of compact sets, so fix some $z \in Z$. Depending on the first symbol of $z$ put: 
$$
U = \begin{cases}\Int B'_{n_0}(x_0,\eps_0) \text{ if }z_0=l\\ \Int B'_{\hn_0}(\hx_0,\eps_0)\text{ if } z_0=r\end{cases}
$$
and define:
$$
k_i = \begin{cases}n_i \text{ if }z_i=l\\ \hn_i\text{ if }z_i=r \end{cases}, i \geq 1.
$$
and 
$$B'_{i} = \begin{cases}B'_{k_i}(x_i,\eps_i) \text{ if }z_i=l\\ B'_{k_i}(\hx_i,\eps_i) \text{ if }z_i=r\end{cases}, i\geq 1.$$ 
Let $\tx_0$ be an arbitrary point from $U$. The backward branch $\{ \tx_j \}_{j\leq 0}$ starting at $\tx_0$ is defined as follows:
\begin{align*}
	\tx_j &\in B'_{m}\text{ such that }f^{k_{m}}(\tx_j) = \tx_{j+k_{m}} \text{ for } j=-\sum_{s=1}^mk_s, m\geq 1, \\
	\tx_{j} &=f(\tx_{j-1}) \text{ for } -\sum_{s=1}^{m}k_s < j < -\sum_{s=1}^{m-1}k_s, m> 1, \text{ or } -k_1<j<0.
\end{align*}
By the definition of $Z$ we have $f^{k_j}(B'_j)\supset B'_{j-1}$ for every $j\geq 1$, so $\tx_j$ is well defined for $j< 0$. 
	
To prove that $\alpha(\{\tx_j\}_{j\leq 0}) \supset \omega(y)$ fix a point $q \in \omega(y)$ and an integer $m~>~0$. Let $b \in \{ x_m,\hx_m \}$ be the point such that $B'_m=B'_{k_m}(b,\eps_m)$. As $\omega(y) \subset \bigcup_{i=0}^{k_m-1}B(f^i(b),\eps_m)$ we can find an integer $0\leq s<k_m$ such that $d(f^s(b),q)<\eps_m$. Simultaneously $d(f^s(b),f^s(\tx_j))<\eps_m$, for $j=-\sum_{i=1}^mk_i$, since $\tx_j\in B'_m$ by the definition of $\{ \tx_j \}_{j\leq 0}$. Altogether we have $d(\tx_{j+s},q)<2\eps_m$ so indeed  $\alpha(\{\tx_j\}_{j\leq 0}) \supset \omega(y)$.

On the other hand, there exists an increasing sequence $\{s_k\}_{k \in \mathbb{N}}$ such that the orbit $\Orb(f^{s_k}(y))$ contains $f^l(x_k)$ for all $0\leq l <n_k$ and $f^l(\hat{x}_k)$ for all $0\leq l<\hat{n}_k$ and the  orbit $\Orb(f^{s_k}(y))$ is $\eps_k$-close from $\omega(y)$.
Therefore for each $m \in \mathbb{N}$ there is $N\geq 0$ such that $\{\tx_j\}_{j\leq -N}\subset B(\overline{\Orb(f^{s_m}(y))},\eps_m)$ which yields that for any $m$
we have $\alpha(\{\tx_j\}_{j\leq 0})\subset B(\omega(y),2\eps_m)$ so $\alpha(\{\tx_j\}_{j\leq 0})\subset \omega(y)$ indeed.
\end{proof}

\begin{rem}\label{rm:ArbLargeConst}
	In the proof of Lemma \ref{lem:MixInfinite} by continuity of map $f$ and the proper choice of the sequence $\{\eps_n\}_{n \in \mathbb{N}}$ with $\eps_n$ decreasing sufficiently fast, for any $\beta>0$ and an arbitrarily large $s \in \mathbb{N}$, we will find a point  following the orbits of each $x_n$ and $\hx_n$ for $s$ iterations at distance at most $\beta$. To achieve this, numbers $n_k,\hn_k$ in the construction must be very large, to overpass any fixed $s$, not only greater that $N_k$.  
\end{rem}

\begin{thm}\label{thm:alphaMixMap}
	Let $f\colon G\to G$ be a mixing graph map. For every $z \in G\setminus \mathcal{I}(f)$ and every $y \in G$ there exists a backward branch $\{z_j\}_{j\leq 0}$ such that $z_0=z$ and $\alpha(\{z_j\}_{j\leq 0})=\omega(y)$.
\end{thm}
\begin{proof}
 Depending on the cardinality of $\omega(y)$ we apply Lemma \ref{lem:MixPer} or Lemma \ref{lem:MixInfinite} to get an open set $U\subset G$ such that for every $\tx_0 \in U$, there is a backward branch $\{\tx_j\}_{j\leq 0}$ with $\alpha(\{\tx_j\}_{j\leq 0})=\omega(y)$. Let $z\in G\setminus\mathcal{I}(f)$. Then there is $\eps>0$ such that $z\in G\setminus  B(\mathcal{I}(f),\eps)$.  By \cite[Theorem 4.6]{HKO}, we can find $k>0$ such that $z\in  \Int f^k(U)$ and thus there exists a preimage $z'\in U$ such that $z=f^k(z')$. Since $z'$ is from $U$ we can find a backward branch $\{\tz_j\}_{j\leq 0}$ with $\tz_0=z'$ and $\alpha(\{\tz_j\}_{j\leq 0})=\omega(y)$. The desired backward branch $\{z_j\}_{j\leq 0}$ starting at $z$ has the form $z_0=z=f^k(z'), z_{-1}=f^{k-1}(z'),\ldots, z_{-k+1}=z'=\tz_0, z_{-k}=\tz_{-1},z_{-k-1}=\tz_{-2},\ldots. $
\end{proof}

\begin{rem}
By \cite[Theorem 4.6]{HKO} we know that inaccessible points are periodic and the set $\mathcal{I}(f)$ is finite and backward invariant, so the only $\alpha$-limit sets of inaccessible points are periodic orbits contained in $\mathcal{I}(f)$.
\end{rem}

Having proved that for mixing map every $\omega$-limit set in $G$ is an $\alpha$-limit set of some backward branch, the natural question is whether it is also true that every $\alpha$-limit set of a backward branch in $G$ is the $\omega$-limit set of some point from $G$ at the same time. The answer to that problem is given below.
\begin{thm}\label{thm:mixingOmega}
	Let $f\colon G\to G$ be the mixing graph map. Then for every backward branch $\{x_j\}_{j\leq 0}\subset G$ the set $\alpha(\{x_j\}_{j\leq 0})$ is equal to an $\omega$-limit set of some point in $G$.
\end{thm}
\begin{proof}
	If $\alpha(\{x_j\}_{j\leq 0})$ is finite then by Lemma \ref{lem:alphaPer} it is a periodic orbit of a point $p \in G$ and obviously $\omega(p) =\alpha(\{x_j\})_{j\leq 0}$.
Assume  $\alpha(\{x_j\}_{j\leq 0})$ is infinite. Let $a\in \alpha(\{x_j\}_{j\leq 0})$ be an accumulation point of  $\alpha(\{x_j\}_{j\leq 0})$ and $J=[a,b]$ be an arc such that  the set $\Lambda = \alpha(\{x_j\}_{j\leq 0})\cap[a,b]$ accumulates on $a$. We may assume that $J\setminus \{a,b\}$ is an open free arc, that means $(J\setminus \{a,b\})\cap Br(G)=\emptyset$.	
 We  show that $\alpha(\{x_j\}_{j\leq 0})$ is approximated by an infinite sequence of periodic orbits. Let $\{\eps_n\}_{n\geq 0}$ be a sequence of positive numbers such that $\eps_0 = \frac12\diam J$ and $\eps_{n+1}<\frac12\eps_n$ for every $n\geq 0$. 
 Let $\{\eta_n\}_{n\geq 0}$ be the sequence of constants from Lemma \ref{lem:10.5} such that for every $x \in G$ and every $k > 0$ we have:
$$
0<\eta_n\leq \diam f^k(B'_k(x,\eps_n)).
$$
For every $k>0$ fix $l_k,m_k,r_k \in \Lambda$ such that $d(l_k,r_k)<\frac{\eta_k}{4}$ and: 
\begin{align*}
a<l_1&<m_1<r_1,\\
a<l_k&<m_k<r_k<l_{k-1} \text{ for }k>1.
\end{align*}
Put $\gamma_k = \frac18\min\{ d(m_k,l_k),d(m_k,r_k) \}$ and let $N_k=N(\eps_k,\gamma_k)$ be the constant from  Lemma \ref{lem:10.4} implying that $B'_{n}(x,\eps_k)\subset B(x,\gamma_k)$ for every $x \in G$ and $n>N_k$. Take a point $y_k \in \{x_j\}_{j\leq 0}$ from $\gamma_k$-neighborhood of $m_k$ with the property that all points from the backward branch preceeding $y_k$ are within $\eps_k$-distance from $\alpha(\{x_j\}_{j\leq 0})$, that is:
\begin{equation}\label{alphaBranchDist}
	\dist(x_i,\alpha(\{x_j\}_{j\leq 0}))<\eps_k \text{ for }i<j_k\text{ where }y_k = x_{j_k}.
\end{equation}
  Choose $n_k,\hn_k >N_k$ for which there exists $z_k,\hz_k \in \{x_j\}_{j\leq 0}$ such that $d(l_k,z_k)< \gamma_k$, $d(r_k,\hz_k)< \gamma_k$ and $f^{n_k}(z_k) = f^{\hn_k}(\hz_k)=y_k$ and note that increasing $N_k$ when necessary, we may assume the following:
 \begin{equation}\label{BowenBallUnion}
\bigcup_{i=0}^{N_k-1}B(x_{j_k-i},\eps_k)\supset \alpha({\{x_j\}_{j\leq 0}}).
\end{equation}
By the  definition of $\eta_k$ and $N_k$ we have:
\begin{align*}
	\eta_k &\leq \diam f^{n_k}(B'_{n_k}(z_k,\eps_k)),\\
	\eta_k &\leq \diam f^{\hn_k}(B'_{\hn_k}(\hz_k,\eps_k)).
\end{align*}
and $B'_{n_k}(z_k,\eps_k)\subset B(z_k,\gamma_k)$ and $B'_{\hn_k}(\hz_k,\eps_k)\subset B(\hz_k,\gamma_k)$. Moreover, $y_k$ is the element of both $ f^{n_k}(B'_{n_k}(z_k,\eps_k))$ and $f^{\hn_k}(B'_{\hn_k}(\hz_k,\eps_k))$. Taking it all into consideration we get the following:
\begin{equation}\label{est3}
	\begin{split}
		d(y_k,m_k)+d(m_k,l_k)+d(l_k,z_k)+&\diam B'_{n_k}(z_k,\eps_k)\\&<\frac{\eta_k}{32}+\frac{\eta_k}{4}+\frac{\eta_k}{32}+\frac{\eta_k}{16} < \frac{\eta_k}{2},\\
		d(y_k,m_k)+d(m_k,r_k)+d(r_k,\hz_k)+&\diam B'_{\hn_k}(\hz_k,\eps_k)\\&<\frac{\eta_k}{32}+\frac{\eta_k}{4}+\frac{\eta_k}{32}+\frac{\eta_k}{16} < \frac{\eta_k}{2},
	\end{split}
\end{equation}
which implies that $f^{n_k}(B'_{n_k}(z_k,\eps_k))$ covers either $(B'_{n_k}(z_k,\eps_k))$ or $(B'_{\hn_k}(\hz_k,\eps_k))$. In the first case there exists a point of period $n_k$ inside $(B'_{n_k}(z_k,\eps_k))$. In the second case we may take $f^{\hn_k}(B'_{\hn_k}(\hz_k,\eps_k))$ which, by  (\ref{est3}), covers either $B'_{n_k}(z_k,\eps_k)$ or $B'_{\hn_k}(\hz_k,\eps_k)$. Then there exists a point of period $n_k+\hn_k$ in $B'_{n_k}(z_k,\eps_k)$ or a point of period $\hn_k$ inside $B'_{\hn_k}(\hz_k,\eps_k)$. Regardless of the case we found at least one periodic point in $\eps_k$-neighborhood of the backward branch $\{x_j\}_{j\leq 0}$. Some iteration of this periodic point denoted by $p_k$ is contained in $B_{N_k}(x_{j_K-N_k},\eps_k)$. Denote the period of $p_k$ by $d_k$. 
The construction results in a set $\{p_k\}_{k> 0}$ of periodic points and an increasing sequence $\{d_k\}_{k> 0}$ of their periods. It assures that for $k>0$ the orbit of each $p_k$ follows some finite segment starting at $x_{j_k-N_k}$ of length $N_k$ of the backward branch $\eps_k$-close, while at the same time all points from that segment stay $\eps_k$-close to the set $\alpha(\{x_j\}_{j\leq 0})$  by (\ref{alphaBranchDist}). Therefore by (\ref{alphaBranchDist}) and (\ref{BowenBallUnion}) we get that for every $k>0$ and every $q \in \alpha(\{x_j\}_{j\leq 0})$ there exists some $i<N_k$ such that:
$$
d(f^i(p_k),x_{j_k-i})+d(x_{j_k-i},q)<2\eps_k.
$$
Combining it with (\ref{alphaBranchDist}) we get that $d_H(\Orb(p_k),\alpha(\{x_j\}_{j\leq 0}))<2\eps_k$ for $k>0$. By \cite[Theorem 3.1]{MShao} we know that the set of all $\omega$-limit sets is closed, so there exists a point $y \in G$ such that the sequence of orbits $\{\Orb(p_k)\}_{k>0} $ converges to  $\omega(y)$ with respect to Hausdorff metric, which gives $\alpha(\{x_j\}_{j\leq 0})=\omega(y)$ completing the proof. 
\end{proof}
\section{Relation of $\alpha$-limit sets of backward branches to maximal $\omega$-limit sets}
We will use the notation from a series of papers by A. Blokh \cite{B1,B2,B3}. In Blokh's papers, a ``graph'' (also called a one-dimensional branched manifold) is not assumed to be connected, and is actually a finite union of graphs with respect to the definition of a~graph we use in the present paper. We reformulate his results for our purposes where necessary in a similar way as authors in \cite{SR}. Therefore we will provide references from both \cite{B1,B2,B3} and \cite{SR}.

A subgraph $K$ of $G$ is called \emph{periodic of period $k$} or \emph{$k$-periodic} if $K, f(K), . . . , f^{k-1}(K)$ are pairwise disjoint and $f^k(K) =K.$ If, instead of $f^k(K) =K$, it is known only that $f^k(K)\subseteq K$, the subgraph $K$ is called \emph{weakly $k$-periodic}. Then the set $\Orb(K) =\cup_{i=0}^{k-1} f^i(K)$ is called a \emph{$k$-cycle of graphs} if $K$ is $k$-periodic and a \emph{weak $k$-cycle of graphs} if $K$ is weakly $k$-periodic. We will write just \emph{cycle of graphs} and \emph{weak cycle of graphs} when the period $k$ is not relevant.

We start this section with simple facts about cycles of graphs containing an infinite $\alpha$-limit set of a backward branch.
\begin{lem}\label{lm:cycle}
Let $f\colon G \to G$ be a graph map and $M\subseteq G$ be a weak $n$-cycle of graphs. Then there is an $n$-cycle of graphs $\hat{M}\subseteq M$. If $M$ contains an infinite $\alpha$-limit set of a backward branch $\alpha(\{y_i\}_{i\leq 0})$ then $\hat{M}$ is non-degenerate and $\alpha(\{y_i\}_{i\leq 0})\subseteq\hat{M}\subseteq M$.
\end{lem}
\begin{proof}
Let $M=\Orb(K)$, where $K$ is an $n$-periodic subgraph of $G$. The set $\hat{K}:=\cap_{k\geq 0}f^{kn}(K)$ is non-empty, compact and connected since it is a decreasing intersection of non-empty compact connected components, and $f^n(\hat{K})=\hat{K}$. If there is an infinite $\alpha(\{y_i\}_{i\leq 0})\subset M$, then we have $\alpha(\{y_i\}_{i\leq 0})\subseteq \Orb(\hat{K})$ and since $\alpha(\{y_i\}_{i\leq 0})$ is infinite, $\hat{K}$ is non-degenerate. The set $\hat{M}=\Orb(\hat{K})$ is an $n$-cycle of graphs.
\end{proof}
Let $f\colon G\to G$ be a graph map and $\alpha(\{y_i\}_{i\leq 0})\subseteq G$ be an infinite $\alpha$-limit set of a backward branch. Then define $$C(\alpha(\{y_i\}_{i\leq 0})) :=\{X |X\subseteq G \text{ is a cycle of graphs and } \alpha(\{y_i\}_{i\leq 0})\subseteq X\}.$$ Since the graph $G$ is weakly 1-periodic and $\alpha(\{y_i\}_{i\leq 0})\subseteq G$, Lemma \ref{lm:cycle} implies that $C(\alpha(\{y_i\}_{i\leq 0}))$ is never empty.
\begin{lem}\label{lm:intersect_cycle}
Let $f\colon G\to G$ be a graph map and $\alpha(\{y_i\}_{i\leq 0})\subseteq G$ be an infinite $\alpha$-limit set of a backward branch. Let $X,Y\in C(\alpha(\{y_i\}_{i\leq 0}))$. Then there is $Z\subset X\cap Y$ which satisfies $Z\in C(\alpha(\{y_i\}_{i\leq 0}))$ and $Z$ has period not smaller than the maximum of periods of $X$ and $Y$.
\end{lem}
\begin{proof} 
Since $\alpha(\{y_i\}_{i\leq 0})$ is infinite the intersection of $\alpha(\{y_i\}_{i\leq 0})$ with some connected component of $X$ (resp. $Y$) is infinite. In fact, every connected component of $X$ (resp. $Y$) contains infinite subset of $\alpha(\{y_i\}_{i\leq 0})$ since $\alpha(\{y_i\}_{i\leq 0})$ is strongly invariant and the preimage of an infinite set has to be infinite. Let $Z_1,\ldots,Z_n$ denote all the connected components of $X\cap Y$ intersecting $\alpha(\{y_i\}_{i\leq 0})$. For every $i\in\{1,\ldots,n\}$, there is $j\in\{1,\ldots,n\}$ such that $f(Z_i)\subseteq Z_j$ since $f(Z_i)$ is included in some component of $X\cap Y$ and meets $f(\alpha(\{y_i\}_{i\leq 0}))=\alpha(\{y_i\}_{i\leq 0})$. Therefore $Z_i$ is weakly periodic with the period not greater than $n$. The set $\alpha(\{y_i\}_{i\leq 0})$ is internally chain transitive by \cite[Lemma 2.1]{HSZ} and thus $\hat{Z}=\{Z_1,\ldots,Z_n\}$ is one weak $n$-cycle of graphs, i.e. it cannot split into a few disjoint cycles. By Lemma~\ref{lm:cycle} there is an $n$-cycle of graphs $Z\subset \hat{Z}$ such that $\alpha(\{y_i\}_{i\leq 0})\subset Z$, and clearly period of $Z$ cannot decrease.	
\end{proof}

\begin{lem}\label{lm:mincycle}
Let $f\colon G\to G$ be a graph map and $\alpha(\{y_i\}_{i\leq 0})\subseteq G$ be an infinite $\alpha$-limit set of a backward branch such that the periods of the cycles in $C(\alpha(\{y_i\}_{i\leq 0}))$ are bounded. There exists a cycle of graphs $X\in C(\alpha(\{y_i\}_{i\leq 0}))$ such that  $X\subseteq Y$ for every $Y\in C(\alpha(\{y_i\}_{i\leq 0}))$.
\end{lem}
\begin{proof}
Let $j$ be the maximal period of the cycles in $C(\alpha(\{y_i\}_{i\leq 0}))$ and by $C_j\subseteq C(\alpha(\{y_i\}_{i\leq 0}))$ denote the family of $j$-cycles of graphs containing the set $\alpha(\{y_i\}_{i\leq 0})$. We will show that there exists $X\in C_j$, such that, for every $\hat{X}\in C_j$, $\hat{X}\subseteq X$ implies $\hat{X}=X$. Let $(Y_{\lambda})_{\lambda\in\Gamma}$ be a totally ordered family in $C_j$ (that is, all elements in $\Gamma$ are comparable and, if $\lambda\leq \mu$, then $Y_{\lambda}\subseteq Y_{\mu}$). Then $Y=\cap_{\lambda\in\Gamma}Y_{\lambda}$ is compact and has $j$ connected components because this is a decreasing intersection of $j$-cycles, and $f(Y) =Y$. Moreover, $\alpha(\{y_i\}_{i\leq 0})  \subseteq Y$ and $Y$ is non-degenerate since $\alpha(\{y_i\}_{i\leq 0})$ is infinite (at least one component of $Y$ is non-degenerate and, by continuity of $f$, every component of $Y$ is non-degenerate). Hence $Y\in C_j$. Thus Zorn's Lemma applies, and there exists a~minimal (with respect to inclusion) element $X\in C_j$ that is, for every $\hat{X}\in C_j$, $\hat{X}\subseteq X$ implies $\hat{X}=X$.

 Let $Y\in C(\alpha(\{y_i\}_{i\leq 0}))$. Then by Lemma \ref{lm:intersect_cycle} there is $X\cap Y\supset Z\in C(\alpha(\{y_i\}_{i\leq 0}))$ which has period greater than or equal to the period of $X$ . On the other hand, the period of $Z$ is at most $j$ by the definition. Hence $Z\in C_j$. Then $Z=X$ by the minimality of $X$, i.e., $X\subseteq Y$.
\end{proof}

 A \emph{generating sequence} or a \emph{sequence generating a solenoidal set} is any nested sequence of cycles of graphs $M_1\supset M_2\supset\cdots$ for $f$ with periods tending to infinity. The intersection $Q=\bigcap_n M_n$ is automatically closed and strongly invariant, i.e. $f(Q)=Q$, and any closed and strongly invariant subset of $Q$ (including $Q$ itself) will be called a \emph{solenoidal set}. Blokh showed that $Q$ contains a perfect minimal set $Q_{min}=Q\cap\overline{Per f}$ such that $Q_{min}=\omega(x)$, for all $x\in Q$, and a maximal $\omega$-limit set (with respect to inclusion) $Q_{max}$ such that $Q_{max}=Q\cap \omega(f)$~\cite[Theorem 1]{B1}. \\

If $x$ is a point of a graph $G$, then by a \emph{side $T$ of the point $x$} we mean a family of open, non-degenerate arcs $\{V_T (x)\}$ containing no branching points, with one endpoint at $x$ and such that if $V_T(x)\in T, \hat{V}_T(x)\in T$, then either $V_T(x)\subseteq \hat{V}_T(x)$ or $\hat{V}_T(x)\subseteq V_T(x)$. Members of the family $T$ are called $T$-sided neighborhoods of $x$.\\
Let $f\colon G\to G$ be a graph map and $M\subset G$ be a cycle of graphs. For every $x\in M$, we define the \emph{prolongation set of $x$ with respect to $f|_M$}: 
$$P_M(x,f)=\bigcap_U\overline{\bigcup_{i=1}^{\infty}f^i(U)},$$ 
where $U$ is a relative neighborhood of $x$ in $M$. If it is clear which map is considered, then we will write $P_M(x)$, if $M=G$ then we will write $P(x,f)$ or $P(x)$. Obviously, $P(x)$ is an invariant closed set and the map $f|_{P(x)}$ is surjective whenever $x$ is a non-wandering point. Similarly, we define the \emph{prolongation set of $x$ with respect to a side $T$}:
 $$P^T_M(x,f)=\bigcap_{V_T(x)}\overline{\bigcup_{i=1}^{\infty}f^i(V_T(x))},$$
where $V_T(x)$ is a relative $T$-sided neighborhood of $x$ in $M$. We will call an arc $V\subseteq G$ \emph{non-wandering} if there is an integer $m\geq 1$ such that $f^m(V)\cap V\neq\emptyset$. It is easy to see that if every $V_T(x)\in T$ is non-wandering then $f|_{P^T(x)}$ is surjective.
\begin{lem}\label{lm:prolongation}
 Let $f\colon G\to G$ be a graph map and $T$ be a side of a point $x\in G$. If every set $V_T(x)\in T$ is non-wandering then $P^T(x)$ is one of the following:
\begin{itemize}
\item $P^T(x)$ is a periodic orbit,
\item $P^T(x)$ is a cycle of graphs,
\item $P^T(x)$ is a solenoidal set $Q$.
 \end{itemize}
\end{lem}
\begin{proof}
By assumptions, for every $V_T(x)\in T$, there is $m\geq 1$ such that $f^m(V_T(x))\cap V_T(x)\neq\emptyset$. Clearly, the set $J_k=\cup_{i=1}^{\infty}f^{mi+k}(V_T(x))$ is connected, for $0\leq k<m$. Thus the set $\cup_{k=0}^{m-1}\overline{J_k}=\overline{\Orb(V_T(x))}$ has finitely many components. Let $I\supset V_T(x)$ be a component of $\overline{\Orb(V_T(x))}$ and $n$ be the minimal integer such that $f^n(V_T(x))\cap V_T(x)\neq\emptyset$. Then $f^n(I)\subseteq I$ and $\overline{\Orb(V_T(x))}$ is a weak cycle of graphs.
Let us choose a family of arcs $\{W_m\}_{m=1}^{\infty}$ so that $W_m\in T$, $W_m\supset W_{m+1}$ and $\lambda (W_m)\to 0$ as $m\to\infty$, where $\lambda(A)$ denotes the length of the arc $A$. By the previous reasoning, $K_m:=\overline{\Orb(W_m)}$ is a weak cycle of graphs, for every $m\geq 1$. Then $K_m\supset K_{m+1}$ and $P^T(x) =\cap_{m\geq 1} K_m$. If periods of $K_m$ are bounded and the intersection is non-degenerate then $P^T(x)$ is a cycle of graphs, since $f|_{P^T(x)}$ is surjective.  If periods of $K_m$ are bounded and the intersection is degenerate then $P^T(x)$ is a periodic orbit. If periods of $K_m$ are unbounded then we can find a generating sequence of cycles of graphs $K'_1\supset K'_2\supset\cdots$, where $K'_m:=\cap_{k\geq 0}f^{k}(K_m)$, such that $P^T(x) =\cap_{m\geq 1} K'_m=Q$ is a solenoidal set.
\end{proof}

Let $M\subset G$ be a cycle of graphs. We define the following sets:
$$E(M,f)=\{x\in M: P_M(x,f)=M\}$$
and
$$E_S(M,f)=\{x\in M: \text{there is a side $T$ such that } P^T_M(x,f)=M\}.$$ 
Clearly, $E_S(M,f)\subseteq E(M,f)$. These sets are closed and invariant. If $E(M,f)$ is infinite then, by~\cite[Theorem 2]{B1}, $E_S(M,f)= E(M,f)$. In general, $E_S(M,f)\neq E(M,f)$ and $f(E(X, f))\neq E(X, f)$. See the following example from \cite{SR}.
\begin{exmp}
Let $\mathbb{S}$ be a circle and decompose $\mathbb{S}$ as the union of "western half-circle" and "eastern half-circle". Let $f$ restricted to any of these half-circles be topologically conjugate to the tent map, the "south pole" of $\mathbb{S}$ being a fixed point of $f$ and the "north pole" being mapped to the "south pole". Then $E(\mathbb{S}, f)$ consists of the two "poles" but $f(E(\mathbb{S}, f))$ is a singleton containing just the "south pole" and $E_S(\mathbb{S},f)$ is empty set.
\end{exmp}

\begin{thm}\cite{B1,B2}\label{E}
Let $M\subset G$ be a cycle of graphs such that $E_S(M,f)$ is non-empty. If $E_S(M,f)$ is finite then it is a periodic orbit. Otherwise, $E_S(M,f)= E(M,f)$ and it is an infinite maximal $\omega$-limit set.
\end{thm}
Let $E(M,f)$ be the infinite maximal $\omega$-limit set from Theorem \ref{E}. Then we say that $E(M,f)$ is a \emph{basic set} if $Per (f)\cap M\neq \emptyset$ and we denote it by $D(M,f)$, while for $Per (f)\cap M= \emptyset$ we say that $E(M,f)$ is a \emph{circumferential set} and we denote it by $S(M,f)$. We will write just $D(M)$ and $S(M)$ in the case where $f$ is clear from the context.
\begin{rem}\label{rm:mincyc} The set $D(M)$ (resp. $S(M)$) is contained in a minimal (with respect to inclusion) cycle of graphs if the periods of the cycles of graphs from the family $C(D(M))$ (resp. $C(S(M))$) are bounded according to Lemma \ref{lm:mincycle}. It was shown in \cite[Remark 17] {SR} that for both $D(M)$ and $S(M)$ this is the case and the minimal cycle of graphs containing $D(M)$ (resp. $S(M)$) is exactly $M$.

\end{rem}
\begin{thm}\label{prop:max}
Let $f\colon G\to G$ be a graph map and $\{y_i\}_{i\leq 0}$ be a backward branch starting at a point $y\in G$. Then $\alpha(\{y_i\}_{i\leq 0})$ is contained in a maximal $\omega$-limit set.
\end{thm}
\begin{proof}
If $\alpha(\{y_i\}_{i\leq 0})$ is a periodic orbit, then it is an $\omega$-limit set and therefore $\alpha(\{y_i\}_{i\leq 0})$ is contained in a maximal $\omega$-limit set (recall that every $\omega$-limit set of a graph map is contained in a maximal one by Mai and Shao \cite{MaiShao}). If $\alpha(\{y_i\}_{i\leq 0})$ is not a periodic orbit, then $\{y_i\}_{i\leq 0}$ has to accumulate at every point $x\in \alpha(\{y_i\}_{i\leq 0})$ from at least one side $T_x$. For every $x\in \alpha(\{y_i\}_{i\leq 0})$, the prolongation set $P^{T_x}(x)$ contains $\{y_i\}_{i\leq 0}$ and since $P^{T_x}(x)$ is a closed invariant set, $P^{T_x}(x)\supseteq \alpha(\{y_i\}_{i\leq 0})\cup \{y_i\}_{i\leq 0}\cup \overline{\Orb(y)}$. By Lemma \ref{lm:prolongation}, $P^{T_x}(x)$ is either a~cycle of graphs or a solenoidal set $Q(x)$. In the latter case, $Q(x)\supset\alpha(\{y_i\}_{i\leq 0})$ and, by results from \cite{SXL}, $\omega(f)\supset\alpha(\{y_i\}_{i\leq 0})$, therefore $\alpha(\{y_i\}_{i\leq 0})$ is contained in the $\omega$-limit set $Q_{max }=Q(x)\cap \omega(f)$. Recall that $Q_{max}$ is a maximal $\omega$-limit set by \cite[Theorem 1]{B1}. If there is no $x\in \alpha(\{y_i\}_{i\leq 0})$ such that $P^{T_x}(x)$ is a solenoidal set, then $P^{T_x}(x)$ is a cycle of graphs for every $x\in \alpha(\{y_i\}_{i\leq 0})$. The set $\alpha(\{y_i\}_{i\leq 0})$ is infinite and thus we can define the family $C(\alpha(\{y_i\}_{i\leq 0}))$. The next step of the proof depends on whether the periods of cycles of graphs in $C(\alpha(\{y_i\}_{i\leq 0}))$ are bounded or unbounded.

We show that if the periods of cycles of graphs in $C(\alpha(\{y_i\}_{i\leq 0}))$ are unbounded then there is a sequence of cycles of graphs $\{X_i\}_{i=1}^{\infty}$ with strictly increasing periods generating a solenoidal set $Q\supset \alpha(\{y_i\}_{i\leq 0})$  and therefore $\alpha(\{y_i\}_{i\leq 0})$ is again contained in a maximal solenoidal set $Q_{max }=Q\cap \omega(f)$. By the assumption there exists a sequence $\{Y_n\}_{n=1}^{\infty}$ of cycles of graphs in $C(\alpha(\{y_i\}_{i\leq 0}))$ with strictly increasing periods $\{l_n\}_{n=1}^{\infty}.$ We define inductively a sequence $\{Y'_n\}_{n=1}^{\infty}$ as follows. Let $Y'_1=Y_1$. If $Y'_n$ is already defined then, according to Lemma \ref{lm:intersect_cycle}, there exists a $l'_{n+1}$-cycle of graphs $Y'_{n+1}$ such that $\alpha(\{y_i\}_{i\leq 0})\subseteq Y'_{n+1}\subseteq Y'_n\cap Y_{n+1}$ and $l'_{n+1}\geq l_{n+1}$. Finally choose a subsequence $\{n_i\}_{i=1}^{\infty}$ such that $l'_{n_i+1}> l'_{n_i}$ for all $i\geq 1$ and set $X_i:=Y'_{n_i}.$

 If the periods of cycles of graphs in $C(\alpha(\{y_i\}_{i\leq 0}))$ are bounded then by Lemma \ref{lm:mincycle} there exists an element $X\in C(\alpha(\{y_i\}_{i\leq 0}))$ such that $X\subseteq Y$ for every $Y\in C(\alpha(\{y_i\}_{i\leq 0}))$. Fix $x\in \alpha(\{y_i\}_{i\leq 0})$. We assumed that $P^{T_x}(x)\in C(\alpha(\{y_i\}_{i\leq 0}))$ and thus $P^{T_x}(x)\supset X$. We will show that the prolongation set $P^{T_x}(x)$ coincides with $P_{X}^{T_x}(x)$ and in consequence $P_{X}^{T_x}(x)=X$. Since $X\supseteq \alpha(\{y_i\}_{i\leq 0})$ and $\alpha(\{y_i\}_{i\leq 0})$ is infinite, $\Int(X)\cap\alpha(\{y_i\}_{i\leq 0})$ is nonempty. It follows that $\{y_i\}_{i\leq 0}\cap X$ is infinite and thus $\{y_i\}_{i\leq 0}\subset X$. Therefore $X$ contains the $T_x$-sided neigborhood of $x$ and $P^{T_x}(x)=P_{X}^{T_x}(x)=X$. By Theorem \ref{E} the set $E_S(X,f)$ is finite iff it is a periodic orbit. But we have just showed that $\alpha(\{y_i\}_{i\leq 0})\subseteq E_S(X,f)$ and $\alpha(\{y_i\}_{i\leq 0})$ is not a periodic orbit by the assumption. Therefore $E_S(X,f)$ is an infinite set and,  by Theorem \ref{E}, it is a maximal $\omega$-limit set.
\end{proof}

For any of the above-mentioned infinite maximal $\omega$-limit sets we can find a model with which the $\omega$-limit set is almost conjugated and this almost conjugacy is unique up to the homeomorphism. For basic sets, the model is a mixing map of a cycle of graphs as described in Corollary \ref{cor:modelBS}. 
\begin{defn}\label{df:ac}
Let $f\colon X\rightarrow X$ and $g\colon Y\rightarrow Y$ be two continuous maps of compact metric spaces $X,Y$ and $K\subseteq X$ be a closed invariant set. A continuous surjection $\phi\colon  X\rightarrow Y$ is an almost conjugacy between $f|_K$ and $g$ if $\phi \circ f=g \circ \phi$ and
\begin{enumerate}
\item $\phi(K)=Y,$
\item $\forall y\in Y, \phi^{-1}(y)$ is connected,
\item $\forall y\in Y, \phi^{-1}(y)\cap K=\partial \phi^{-1}(y)$, where $\partial A$ denotes the boundary of $A$.
\end{enumerate}
\end{defn}
If $X,Y$ are graphs or cycles of graphs, the conditions (2) and (3) imply $ \phi^{-1}(y)\cap K$ is a set of endpoints of a subgraph of $X$ and hence a finite set, for every $y\in Y$.
\begin{lem}\label{lem:almostC}
Let $f\colon X\rightarrow X$ and $g\colon Y\rightarrow Y$ be two continuous maps of cycles of graphs $X,Y$ and $K\subseteq X$ be a closed invariant set. If there is an almost conjugacy $\phi$ between $f|_K$ and $g$, then $\phi$ is unique up to the homeomorphism. 
\end{lem}
\begin{proof}
Let $x$ be an arbitrary point from $X$ and $\phi_1$ and $\phi_2$ be almost conjugacies between $f|_K$ and $g$. Then $\phi_1^{-1}(\phi_1(x))$ (respectively, $\phi_2^{-1}(\phi_2(x)))$ is a connected set containing $x$. We will show that $\phi_1^{-1}(\phi_1(x))\equiv \phi_2^{-1}(\phi_2(x))$. Assume the contrary. Then $\phi_1$ is not a constant function on $\Int(\phi_2^{-1}(\phi_2(x)))$ or $\phi_2$ is not a constant function on $\Int(\phi_1^{-1}(\phi_1(x)))$. We can assume the first case without loss of generality. Then there are $y,z\in\Int(\phi_2^{-1}(\phi_2(x)))$ such that $\phi_1(y)\neq\phi_1(z)$. Since $\phi_1^{-1}(\phi_1(y))$ does not contain $z$, we have $\partial\phi_1^{-1}(\phi_1(y))\cap \Int(\phi_2^{-1}(\phi_2(x)))\neq \emptyset$. This is impossible since, by (3) from Definition \ref{df:ac}, $\partial \phi_1^{-1}(\phi_1(y))\subset K$ and $\Int (\phi_2^{-1}(\phi_2(x)))\cap K=\emptyset$.\\
By \cite[Corollary 22.3]{Munk}, the quotient spaces $\phi_1(X)$ and $\phi_2(X)$ are homeomorphic.

\end{proof}

\begin{thm}\cite{B1,SR}\label{thm:modelE}
Let $f\colon G\rightarrow G$ be a graph map and $X\subseteq G$ be a cycle of graphs. Suppose that $E(X,f)$ is infinite. Then there is a transitive map $g\colon Y\rightarrow Y$, where $Y$ is a cycle of graphs, and $\phi\colon X\rightarrow Y$ which almost conjugates $f|_{E(X,f)}$ and $g$.
\end{thm}
A transitive graph map is either totally transitive or it can be decomposed into a totally transitive one. 
\begin{thm}\cite{HKO}\label{thm:trans_decomp}
Let $f\colon G\rightarrow G$ be a transitive graph map. Then exactly one of the following statements holds.
\begin{enumerate}
\item $f$ is totally transitive,
\item there is $k>1$ and non-degenerate subgraphs $G_0,\ldots, G_{k-1}$ of $G$ such that
\begin{enumerate}
\item $G=\cup_{i=0}^{k-1}G_i$,
\item $G_i\cap G_j=End(G_i)\cap End(G_j)$, for $i\neq j$,
\item $f(G_i)=G_{i+1 \mod k}$, for $i=0,\ldots,k-1$,
\item $f^k|_{G_i}$ is totally transitive, for $i=0,\ldots,k-1$.

\end{enumerate}
\end{enumerate}
\end{thm}
It follows from \cite[Corollary 4.3, Theorem 3.2]{HKO2} that every totally transitive map acting on a graph $G$ where $G$ is not the circle is mixing.
Let $\mathbb{S}$ be the unit circle. If $f$ acting on $\mathbb{S}$ is totally transitive and sensitive then $f$ is mixing by \cite[Theorem 4.2, Theorem 3.2]{HKO2}. If it is totally transitive and not sensitive then $f$ is a transitive almost equicontinuous map, by Auslander-Yorke Dichotomy \cite{AY}, and therefore $f$ is a homeomorphism by \cite{AAB}. \\
Since a basic set $D(X)$ contains a periodic orbit, the transitive model map $g$ with which $D(X)$ is almost conjugated contains a periodic orbit as well. By \cite[Proposition 11.1.4, Proposition 11.2.2]{KH}, a homeomorphism acting on $\mathbb{S}$ possessing periodic points is never transitive. Therefore $g$ is not a homeomorphism of the circle and, by the reasoning above, if $g$ is totally transitive then $g$ is mixing. This fact together with Theorem \ref{thm:modelE} and Theorem \ref {thm:trans_decomp} implies the following corollary.
\begin{cor}\label{cor:modelBS}
Let $f\colon G\rightarrow G$ be a graph map and $X\subseteq G$ be a cycle of graphs. Suppose that $D(X)$ is a basic set. Then there is a transitive map $g\colon Y\rightarrow Y$, where $Y$ is a cycle of graphs $Y_0,\ldots Y_{n-1}$ with possibly non-empty intersection in the endpoints, and $\phi\colon X\rightarrow Y$ which almost conjugates $f|_{D(X)}$ and $g$. Moreover, $g^n|Y_i$ is mixing, for $i=0,\ldots,n-1$. The period $n$ of $Y$ is a multiple of the period of $X$ and $Y_i\cap Y_j=End(Y_i)\cap End(Y_j)\neq \emptyset$ iff $i\neq j$ and $i$ and $j$ are congruent modulo the period of $X$.

\end{cor}
The next lemma will help us to transfer some constructions from the model space $Y$ to the basic set $D(X)$ later in Section \ref{sec:PosEnt}.

	\begin{lem}\label{lem:P} Let $X, f, D(X),\phi, g, Y$ be as in  Corollary \ref{cor:modelBS} and let $\delta>0$. Then there is a finite set $P(\delta)=\{p_1,\ldots,p_k\}\subset Y$ and $\gamma>0$ such that, for every neigborhood $U(x)$ of any point $x\in Y\setminus P(\delta)$ with $\diam(U(x))\leq \gamma$ and $U(x)\cap P(\delta)=\emptyset$, we have $\diam(\phi^{-1}(U(x)))<\delta$.
	\end{lem}
	\begin{proof}
		Since $\diam(X)$ is finite, we have at most countably many points $y\in Y$ such that $\phi^{-1}(y)$ is not a singleton and we can arrange them into a sequence $\{p_i\}_{i\geq 1}$.
		We also include as first positions in the sequence all branching points of the graph $Y$.
		Since we have also $\sum_{i=1}^{\infty}\diam((\phi^{-1}(p_i))<\infty$, there is $k\geq 1$ such that $\sum_{i=k+1}^{\infty}\diam((\phi^{-1}(p_i))<\delta/8$
		and $k$ is larger than the number of branching points, so denote $P(\delta) = \{p_1,\dots,p_k\}$. 
		Let $Z=X\setminus \bigcup_{i=1}^k \Int \phi^{-1}(p_i)$. Note that $Z$ is a finite union of graphs, in particular it is compact.
		Let $V_1,\ldots, V_s$ be an open cover of $Z$ by connected sets of diameter $\delta/8$. Let $U_i=\phi(V_i)$ for each $i=1,\dots,s$ and
		let $\gamma=\min_{1\leq i\leq s} \diam U_i$. Fix any $x\in Y\setminus P(\delta)$ with $\diam(U(x))\leq \gamma$ and $U(x)\cap P(\delta)=\emptyset$.
		Take any $p,q\in \phi^{-1}(U(x))$ and consider $\hat{p}=\phi(p), \hat{q}=\phi(q)\in U(x)$. Since $U(x)$ is connected and does not contain branching points, there is $\hat{z}\in U(x)$
		and $i,j$ such that $\hat{p},\hat{z}\in U_i$ and $\hat{q},\hat{z}\in U_j$. But then
		\begin{eqnarray*}
		d(p,q)&\leq& \diam \phi^{-1}(\hat{p})+\diam V_i+\diam \phi^{-1}(\hat{z})+\diam V_j+\diam \phi^{-1}(\hat{q})\\
		&\leq& \frac{5\delta}{8}<\delta.
		\end{eqnarray*}		
The proof is complete.		
	\end{proof}

\section{Zero entropy graph maps}\label{sec:zeroEnt}
We will show that the structure of the family of $\alpha$-limit sets of backward branches for a graph map greatly depends on the entropy of the map. In particular, for zero entropy maps the family of $\alpha$-limit sets of backward branches coincides with the family of minimal sets.
The following well-known theorem shows that graph maps with zero topological entropy do not possess basic sets.
\begin{thm}\cite{M}\label{PosEnt}
Let $f$ be a continuous graph map. Then the following conditions are equivalent:
\begin{enumerate}
\item $h(f)>0$,
\item $f$ has a basic set.
\end{enumerate}
\end{thm}

\begin{thm} \label{thm:min}Let $f$ be a continuous map acting on a graph $G$ with $h(f)=0$. Then a set $L$ is an $\alpha$-limit set of a backward branch $\{x_j\}_{j\leq 0}$ if and only if $L$ is a~minimal set.
\end{thm}
\begin{proof}
Since a minimal set $L$ is closed, for any backward branch $\{x_j\}_{j\leq 0}\subseteq L$ we have  $\alpha(\{x_j\}_{j\leq 0})\subseteq L$. But $\alpha(\{x_j\}_{j\leq 0})$ is a closed invariant set. By minimality of $L$, $\alpha(\{x_j\}_{j\leq 0})=L$.\\
By Blokh's Decomposition Theorem \cite[Theorem 4]{B1} and Theorem \ref{PosEnt}, the maximal $\omega$-limit sets of the system $(G,f)$ are solenoidal sets, circumferential sets and periodic orbits which are maximal $\omega$-limit sets with respect to inclusion. If $L$ is an $\alpha$-limit set of a backward branch $\{x_j\}_{j\leq 0}$ then, by Theorem \ref{prop:max}, $L$ is contained in one of these maximal $\omega$-limit sets. If $L=\alpha(\{x_j\}_{j\leq 0})$ is contained in a periodic orbit, then $L$ coincides with this periodic orbit. \\
Assume that $L$ is contained in a solenoidal maximal $\omega$-limit set $Q_{max}$ and let $M_1\supset M_2\supset \ldots$ be the generating sequence of cycles of graphs with periods tending to infinity such that $Q_{max}\subseteq \cap_n M_n$. Since $L$ is infinite, $\{x_j\}_{j\leq 0}\cap M_n\neq \emptyset$, for every $n$, and, by the invariance of $M_n$, $\{x_j\}_{j\leq 0}\subset M_n$, for every $n$. There is a cycle $M_k$ with period $m(k)$ greater than $\# Br(G)$. Denote by $I$ the connected component of $M_k$ such that $M_k\cap B(G)=\emptyset$. The set $\{x_j\}_{j\leq 0}\cap I$ is infinite and forms a backward branch with respect to $f^{m(k)}$. Therefore $I\cap L$ is an $\alpha$-limit set of the backward branch for the zero entropy interval map $f^{m(k)}|_I$ and, by Theorem 12 from \cite{BDLO}, $I\cap L$ is a perfect set. If $z$ is an isolated point of $L$, then $z$ has a pre-image $\hat{z}$ in $I\cap L$. Since $f$ is continuous and $\hat{z}$ is not isolated in $I\cap L$, there is a neighbourhood $U$ of $\hat{z}$ in $I\cap L$ such that $U$ is eventually mapped on $z$. This implies $f^{m(k)}(U)$ is a singleton and as a consequence $U$ contains a periodic point. But it is impossible, since there are no periodic points in a solenoidal set $Q_{max}$ and it follows that $L$ is a perfect set. By \cite[Theorem 1]{B1}, $Q_{max}$ has at most countable set of isolated points and the set of all limit points of $Q_{max}$ is contained in the minimal set $Q_{min}=Q\cap \overline{Per f}$. Therefore $L\subseteq Q_{min}$ and, by minimality of $Q_{min}$, $L=Q_{min}.$\\
Let $L$ be contained in a circumferential set $S(X,f)$. The following result can be found in \cite[Theorem 3]{B1} or \cite{SR} and we briefly recall it here. Let  $X_1,\ldots,X_n$ be the connected components of $X$. Then, either, for every $i$, $f^k|_{X_i}$ is conjugate to an irrational rotation (and in this case $S(X_i, f^k) =X_i)$, or, for every $i$, there exists a semi-conjugacy $\phi_i$ between $f^k|_{X_i}$ and an irrational rotation which is an almost conjugacy on $f^k|_{S(X_i,f^k)}$. The latter case is called the Denjoy type of $\omega$-limit set and it is described in detail in \cite{MShao}. In both cases, $S(X,f)$ is the unique minimal set of the system $(X,f)$ and all points in $X\setminus S(X,f)$ are wandering \cite[Corollary 4.4]{MShao}. By minimality of $S(X,f)$, $L\subseteq S(X,f)$ implies $L= S(X,f)$.
 
\end{proof}
\begin{rem}\label{rm:minsets}
A minimal set for a graph map $f$ with $h(f)=0$ is either a periodic orbit, a minimal solenoidal set or a circumferential set. Theorem \ref{thm:min} shows that these (and only these) sets can be realized as $\alpha$-limit sets of backward branches for~$f$.
\end{rem}
Denote the family of all $\alpha$-limit sets of backward branches starting at a point $x\in G$ by $\mathcal{A}(x),$
$$\mathcal{A}(x)=\{L\in \mathcal{P}(G): \exists \{x_j\}_{j\leq 0} \text { such that } x_0=x \text { and } L=\alpha(\{x_j\}_{j\leq 0})\}.$$
\begin{cor}
 Let $f$ be a continuous map acting on a graph $G$ with $h(f)=0$ and $x\in G$. Then $\mathcal{A}(x)$ contains at most one infinite set.
\end{cor}
\begin{proof}
Let  $\{x_j\}_{j\leq 0}$, $x_0=x$, be a backward branch with $L:=\alpha(\{x_j\}_{j\leq 0})$ infinite. By Remark \ref{rm:minsets}, $L$ is either a minimal solenoidal set or a circumferential set. Assume the first case. Then there is a generating sequence of cycles of graphs $M_1\supset M_2\supset~\ldots$ with periods tending to infinity such that $L\subseteq Q$, where $Q=\cap_n M_n$. Since $L$ is infinite, the backward branch $\{x_j\}_{j\leq 0}$ intersects the cycle of graphs $M_n$, for every $n$. By the invariance of $M_n$, $x\in M_n$, for every $n$, and $x$ belongs to the solenoidal set $Q$. It is a well known fact that two solenoidal sets $Q=\cap_n M_n$ and $Q'=\cap_n M'_n$ generated by different sequences $M_1\supset M_2\supset \ldots$ and $M'_1\supset M'_2\supset \ldots$ are either identical or disjoint. Since $x\notin Q'$, for any $Q'\neq Q$, and since there is only one minimal set in $Q$, $L$ is the unique minimal solenoidal set in $\mathcal{A}(x)$. 

It is easy to see that, for every circumferential set $S(X)$, where $X$ is the minimal cycle of graphs containing $S(X)$, we have $M\supseteq X$ or $M\cap X=\emptyset$, for every cycle of graphs $M$. Clearly, $M\cap X\subsetneq X$ is impossible since $S(X)$ is the unique minimal set of the system $(X,f)$ and, by Remark \ref{rm:mincyc}, $X$ is the minimal cycle of graphs containing $S(X)$. Therefore $\cap_n M_n\cap X=\emptyset$ and $x\notin X$. Since every backward branch $\{x'_j\}_{j\leq 0}$, $x'_0=x$, has empty intersection with $X$, $S(X)$ does not belong to $\mathcal{A}(x)$ and $L$ is the unique infinite set in $\mathcal{A}(x)$.\\
Assume that $L$ is a circumferential set $S(X)$. Then $x\in X$. For any circumferential set $S(X')$, we have either $X\subseteq X'\wedge X'\subseteq X\implies X=X' $ and $S(X)=S(X')$, or the intersection $X\cap X'$ is empty and $x\notin X'$. Since every backward branch $\{x'_j\}_{j\leq 0}$, $x'_0=x$, has empty intersection with $X'$, for every $X'\neq X$, $S(X)$ is the unique circumferential set in $\mathcal{A}(x)$. By the reasoning above, there is no minimal solenoidal set in  $\mathcal{A}(x)$  and $L$ is the unique infinite set in $\mathcal{A}(x)$.
\end{proof}

In addition to one infinite $\alpha$-limit set, the family $A(x)$ can contain many finite $\alpha$-limit sets.  Every finite $\alpha$-limit set is a periodic orbit by Theorem \ref{thm:min}. In the following example, we will construct $A(x)$ containing a circumferential set and uncountably many periodic orbits.

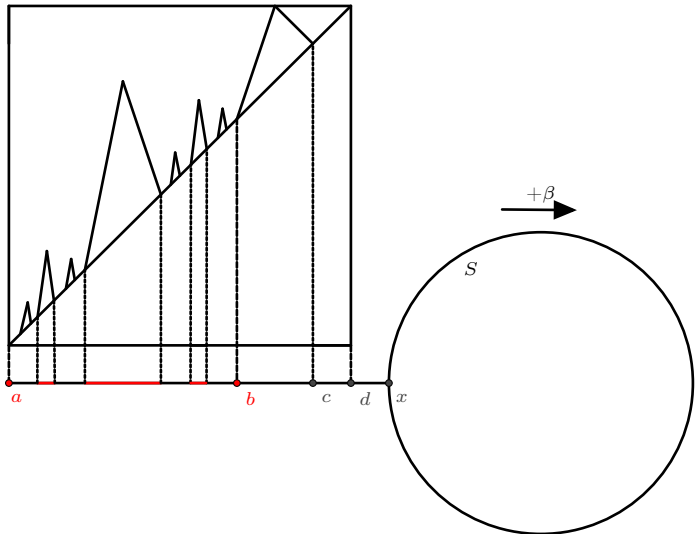
\begin{figure}[h]

\begin{tikzpicture}[scale=0.5,line cap=round,line join=round,>=triangle 45,x=1cm,y=1cm]
\clip(-0.3,-5) rectangle (20.200093106883806,10.228608665848423);
\draw  [line width=1pt] (0,-1)-- (10,-1);
\draw [line width=1pt] (0,0)-- (9,0);
\draw [line width=1pt,dash pattern=on 1pt off 1pt] (8,-1)-- (8,8);
\draw [line width=1pt] (0,9)-- (0,0);
\draw [line width=1pt] (9,9)-- (0,0);
\draw [line width=1pt] (6,6)-- (7,9);
\draw [line width=1pt] (7,9)-- (8,8);
\draw [line width=1pt,color=ffqqqq] (2,-1)-- (4,-1);
\draw [line width=1pt] (4.779375,0)-- (5.20125,0);
\draw [line width=1pt,color=ffqqqq] (0.7481249999999997,-1)-- (1.2075,-1);
\draw [line width=1pt,dash pattern=on 1pt off 1pt] (0.7481249999999997,-1)-- (0.7546875000000002,0.7546875000000002);
\draw [line width=1pt,dash pattern=on 1pt off 1pt] (1.2075,-1)-- (1.1859374999999996,1.1859374999999996);
\draw [line width=1pt,dash pattern=on 1pt off 1pt] (2,-1)-- (2,2);
\draw [line width=1pt,dash pattern=on 1pt off 1pt] (4,-1)-- (4,4);
\draw [line width=1pt,dash pattern=on 1pt off 1pt] (4.779375,-1)-- (4.790625000000001,4.790625000000001);
\draw [line width=1pt,dash pattern=on 1pt off 1pt] (5.20125,-1)-- (5.2078125,5.2078125);
\draw [line width=1pt] (2,2)-- (3,7);
\draw [line width=1pt] (3,7)-- (4,4);
\draw [line width=1pt] (4.790625000000001,4.790625000000001)-- (5,6.5);
\draw [line width=1pt] (5,6.5)-- (5.2078125,5.2078125);
\draw [line width=1pt] (0.7546875000000002,0.7546875000000002)-- (1,2.5);
\draw [line width=1pt] (1,2.5)-- (1.1859374999999996,1.1859374999999996);
\draw [line width=1pt] (5.5,5.5)-- (5.623125000000001,6.27375);
\draw [line width=1pt] (5.623125000000001,6.27375)-- (5.728125,5.728125);
\draw [line width=1pt] (4.260937500000001,4.260937500000001)-- (4.37625,5.11125);
\draw [line width=1pt] (4.37625,5.11125)-- (4.5,4.5);
\draw [line width=1pt] (1.5,1.5)-- (1.63875,2.2893749999999997);
\draw [line width=1pt] (1.63875,2.2893749999999997)-- (1.734375,1.734375);
\draw [line width=1pt] (0.3,0.3)-- (0.495,1.13625);
\draw [line width=1pt] (0.495,1.13625)-- (0.58125,0.58125);
\draw [line width=1pt] (8,0)-- (9,0);
\draw [line width=1pt] (14,-1) circle (4cm);
\draw [line width=1pt] (0,9)-- (9,9);
\draw [line width=1pt] (9,9)-- (9,0);
\draw [line width=1pt] (0,9)-- (0,8);
\draw [line width=1pt,color=ffqqqq] (4.784572729679956,-1)-- (5.187276825757309,-1);
\draw [line width=1pt,dash pattern=on 2pt off 1pt] (6,6)-- (6,-1);
\draw [line width=1pt,dash pattern=on 2pt off 1pt] (0,-1)-- (0,0);
\draw [line width=1pt,dash pattern=on 2pt off 1pt] (9,-1)-- (9,0);
\draw [->,line width=1pt] (12.978935460221383,3.60350613911873) -- (14.948270028275864,3.586381490700866);
\begin{scriptsize}
\draw [fill=ffqqqq] (0,-1) circle (2.5pt);
\draw[color=ffqqqq] (0.2,-1.4) node {$a$};
\draw[color=black] (12.157389189439836,2.0255374996702455) node {$S$};
\draw [fill=uuuuuu] (9,-1) circle (2.5pt);
\draw[color=uuuuuu] (9.37081135537206,-1.4) node {$d$};
\draw [fill=ffqqqq] (6,-1) circle (2.5pt);
\draw[color=ffqqqq] (6.363076550346523,-1.4) node {$b$};
\draw [fill=uuuuuu] (8,-1) circle (2.5pt);
\draw[color=uuuuuu] (8.353489288966363,-1.4) node {$c$};
\draw [fill=uuuuuu] (10,-1) circle (2.5pt);
\draw[color=uuuuuu] (10.343902027586203,-1.4) node {$x$};
\draw[color=black] (14,4) node {+$\beta$};
\end{scriptsize}
\end{tikzpicture}
\caption{A map where the family $\mathcal{A}(x)$ is uncountable.}\label{fig:zeroentropy}
\end{figure}
\begin{exmp}Let $G$ be the union of the circle $S$ and the interval $[a,x]$ as shown on the Figure \ref{fig:zeroentropy}. The map $f\colon G\to G$ is defined such that $f|_S$ is the rotation by an irrational angle $\beta$ and the graph of $f\colon [a,c]\to [a,d]$ is sketched on the Figure \ref{fig:zeroentropy}. The construction of the interval map $f|_{[a,b]}$ was previously used in \cite[Example 4.8]{KMS}. The remaining interval $[c,x]$ is mapped by $f$ into $[c,x]\cup S$ continuously in such a way that $f(c)=c$ and $f(d)=x$. Then $\mathcal{A}(x)$ consists of a middle-third Cantor set of fixed points in the interval $[a,b]$ (drawn by red color on  the Figure \ref{fig:zeroentropy}) and the circumferential set $S$.
\end{exmp}
\section{Positive entropy graph maps}\label{sec:PosEnt}
In this section, we will investigate $\alpha$-limit sets of backward branches which are included in a basic set. By Theorem \ref{PosEnt}, every continuous graph map $f$ with $h(f)>0$ possess a basic set $D(X)$. The main goal is to use the model map $g\colon Y\rightarrow Y$ for the basic set $D(X)$ given by Corollary \ref{cor:modelBS} to obtain a similar result as for mixing graph maps in Theorem \ref{thm:alphaMixMap} and \ref{thm:mixingOmega}. Recall that the model map $g$ is almost conjugate to $f|_{D(X)}$ and $g^n|_{Y_i}$ is mixing on every component $Y_i$ of $Y$.\\
We introduce an equivalence relation on $X$  as follows:
	 $$x\sim y \Leftrightarrow \phi(x)=\phi(y),$$
 where $x,y \in X$ and $\phi$ is the almost conjugacy between $f|_{D(X)}$ and $g$ from Corollary~\ref{cor:modelBS}. The relation $\sim$ is well defined since $\phi$ is unique up to the homeomorphism by Lemma \ref{lem:almostC}. Denote $[x]_{\sim}$ the equivalence class of a point $x\in X$ and $[A]_{\sim}=\bigcup_{x\in A} [x]_{\sim}$, for any $A\subset X$. Obviously $[x]_{\sim}=\phi^{-1}(\phi(x))$ and $[A]_{\sim}=\phi^{-1}(\phi(A))$. By the definition of almost conjugacy, $[x]_{\sim}$ is a subgraph of $X$ such that $End([x]_{\sim})=[x]_{\sim}\cap D(X)$ and $f([x]_{\sim})\subseteq [f(x)]_{\sim}$. The last inclusion ensures that for every backward branch $\{\tx_j\}_{j\leq 0}\subset Y$ constructed with respect to the model map $g$ and starting at $\phi(x)$ there is a backward branch $\{z_j\}_{j\leq 0}\subset X$ with respect to $f$ such that $\phi(z_j)=\tx_j$, for every $j\geq 0$. Unfortunatelly, the oposite inclusion $f([x]_{\sim})\supseteq [f(x)]_{\sim}$ may not hold in general for every $x\in X$. This makes our aim to use the model map $g$ difficult since the backward branch $\{z_j\}_{j\leq 0}$ may not start at $x$ but at some other point of $[x]_{\sim}$. Therefore in Theorem \ref{thm:alphaPosEntMap} we restrict ourselves to the case when $f([x]_{\sim})= [f(x)]_{\sim}$ for every $x\in X$ or, equivalently, $f(\phi^{-1}(y))=\phi^{-1}(g(y))$ for every $y\in Y$.\\
 Let $\mathcal{I}(g^n|_{Y_i})$ be the set of inacessible points of the mixing graph map $g^n:Y_i\to Y_i$ given by Corollary \ref{cor:modelBS}, for every $i=0,1,\ldots, n-1$. Then we define the set of inaccesible points of $X$ as the union of preimages of inacessible points of the model mixing map,
$$\mathcal{I}(X)=\bigcup_{0\le i\leq n-1} \phi^{-1}(\mathcal{I}(g^n|_{Y_i})).$$
\begin{lem}\label{lem:mixingE}
Let $V\subset X$ be a subgraph such that $\phi(V)$ is a non-degenerate subgraph of $Y$. Then $\bigcup_{k=0}^{\infty}f^k(V)\supseteq X\setminus \mathcal{I}(X)$. Consequently, for every point $x\in X\setminus \mathcal{I}(X)$ there is a preimage $z\in f^{-k}(x)\cap V$, for some $k>0$.
\end{lem}
\begin{proof}
Notice that if $\phi(x)\in \Int (\phi(A))$, for some $x\in X$ and $A\subset X$, then $x\in A$. Since $\phi(V)$ is a non-degenerate subgraph of $Y$, there is a component $Y_i$ of $Y$ such that $\phi(V)\cap Y_i$ is non-degenerate. Since $g^n|_{Y_i}$ is mixing, we have by Equation \ref{eq:inac},  $$ \bigcup_{k=0}^{\infty}\Int(\phi(f^{n\cdot k}(V)))=\bigcup_{k=0}^{\infty}\Int(g^{n\cdot k}(\phi(V)))=Y_i\setminus\mathcal{I}(g^n|_{Y_i}).$$
The image of $\phi(V)$ by $g$ is a non-degenerate subgraph of the component $Y_{i+1}$ (otherwise $g^{n\cdot k}(\phi(V))$ is a singleton, for every $k>1$, which is in a contradiction with the equation above) and the same holds for every $g^j(\phi(V))$, $j=0,1,\ldots, n-1$. Again by Equation \ref{eq:inac}, 
$$ \bigcup_{k=0}^{\infty}\Int(\phi(f^{k}(V)))=\bigcup_{j=0}^{n-1}\bigcup_{k=0}^{\infty}\Int(g^{n\cdot k+j}(\phi(V)))=\bigcup_{j=0}^{n-1}Y_j\setminus\mathcal{I}(g^n|_{Y_j})=\phi(X\setminus \mathcal{I}(X)).$$
Therefore $x\in X\setminus \mathcal{I}(X)$ implies $\phi(x)\in \Int(\phi(f^k(V)))$, for some $k>0$, and we have $x\in f^k(V)$.

\end{proof}

\begin{thm}\label{thm:alphaPosEntMap}
 Let $D(X)$ be a basic set such that $f([x]_{\sim})= [f(x)]_{\sim}$, for every $x\in X$. Then, for every $x \in X\setminus \mathcal{I}(X)$ and every $\omega$-limit set $\omega_f(y)$ such that $\omega_f(y)\subset D(X)$ is infinite, there exists a backward branch $\{z_j\}_{j\leq 0}$ such that $z_0=x$ and $\omega_f(y)\subseteq \alpha(\{z_j\}_{j\leq 0})\subseteq[\omega_f(y)]_{\sim}\cap D(X)$.
\end{thm}
\begin{proof}
Assume first that $f|_{D(X)}$ is almost conjugate to a mixing graph map $g\colon Y\rightarrow Y$, i.e. the cycle of graphs $Y$ has only one component and take $y \in D(X)$ such that $\omega_f(y)\subset D(X)$ is infinite. If $y\notin D(X)$ then we can replace it by a point from $[y]_{\sim}\cap D(X)$ since the diameter of sets $[f^i(y)]_{\sim}$ tends to 0 as $i\to\infty$ and $\omega_f(x)=\omega_f(y)$, for every $x\in[y]_{\sim}$.
Note that $\phi(\omega_f(y))=\omega_g(\phi(y))$. Image by $\phi$ of any limit point of $\Orb_f(y)$ is a limit point of $\Orb_g(\phi(y))$ and conversely, by compactness, any limit point of $\Orb_g(\phi(y))$ can be obtained as an image of a limit point in  $\Orb_f(y)$.\\
Below we use the notation from the proof of Lemma \ref{lem:MixInfinite}. We introduce some modification implied by Remark \ref{rm:ArbLargeConst} to the construction in order to recover the desired backward branch $\{z_j\}_{j\leq 0}$. The construction from the proof of Lemma \ref{lem:MixInfinite} applied for $\omega$-limit set $\omega_g(\phi(y))\subset Y$ and any $x$ in an open set $U\subset Y$ gives us the backward brach $\{\tx_j\}_{j\leq 0}$ such that $\tx_0=x$ and $\alpha_g(\{\tx_j\}_{j\leq 0}) = \omega_g(\phi(y))$. The modification in the proof of Lemma \ref{lem:MixInfinite} is as follows. Fix $i\in\mathbb{N}$. Let $P(\epsilon_i)$ be the finite set from Lemma \ref{lem:P}. Since $\Orb(\phi(y))$ is infinite ($\Orb(y)$ is infinite subset of $D(X)$ and $\phi|_{D(X)}$ is finite-to-one), we can find $N>0$ such that $\phi(f^n(y))\cap P(\epsilon_i)=\emptyset$, for $n>N$. In the proof of Lemma \ref{lem:MixInfinite} we constructed a sequence $\{x_i\}_{i \in \mathbb{N}}$ and $\{\hx_i\}_{i \in \mathbb{N}}$ such that $x_i, \hx_i \in \Orb(\phi(y))$ for every $i \in \mathbb{N}$ and associated sequences $\{n_i\}_{i \in \mathbb{N}}$ and $\{\hn_i\}_{i \in \mathbb{N}}$ of times for which the orbits of points in the Bowen ball follows $\epsilon_i-$close the orbit of $x_i$ and $\hx_i$ respectively. We may require that  $x_i,\hx_i\in \Orb(\phi(f^{N+1}(y)))$. Since none of the points $x_i,\hx_i$ or their forward iterates belongs to $P(\epsilon_i)$, assuming that $n_i, \hn_i$ are sufficiently large, we may require that set (which can be arbitrarily small) $B'_{n_i}(x_i,\epsilon_i)$ (resp. $B'_{\hn_i}(\hx_i,\epsilon_i)$) does not contain any points from $P(\epsilon_i)$. Furthermore, if we fix any $s_i$, then by continuity, if $n_i, \hn_i$ are sufficiently large (in practice, much larger than $s_i$), also $g^j(B'_{n_i}(x_i,\epsilon_i))$, (resp. $g^j(B'_{\hn_i}(\hx_i,\epsilon_i))$) does not contain any points from $P(\epsilon_i)$ for $j=0,\ldots, s_i$. By Lemma \ref{lem:P}, if we fix any $q_i$ such that $\phi(q_i)=x_i$ and any $\tilde{q}_i$ such that $\phi(\tilde{q}_i)\in B'_{n_i}(x_i,\epsilon_i)$ then $d(f^j(q_i),f^j(\tilde{q}_i))<\epsilon_i$ for $j=0,\ldots, s_i$. The same holds if $\phi(q_i)=\hx_i$ and any $\tilde{q}_i$ such that $\phi(\tilde{q}_i)\in B'_{\hn_i}(\hx_i,\epsilon_i)$. In particular, we can take $q_i\in \Orb(f^{N+1}(y))$ obtaining that $\cup_{j\leq s_i}B(f^j(\tilde{q}_i),2\epsilon_i)\supset \omega_f(y)$ provided that $s_i$ was sufficiently large. This modification ensures that $ \alpha(\{\tilde{q}_j\}_{j\leq 0})\supset \omega_f(y)$ whenever $\{\tilde{q}_j\}_{j\leq 0}$ is a backward branch such that $\phi(\tilde{q}_j)=\tx_j$, for every $j\leq 0$. On the other hand, $\phi(\omega_f(y))=\omega_g(\phi(y))=\phi(\alpha(\{\tilde{q}_j\}_{j\leq 0}))$ which gives $ \alpha(\{\tilde{q}_j\}_{j\leq 0})\subset [\omega(y)]_{\sim}$. Since $\alpha(\{\tilde{q}_j\}_{j\leq 0})$ is a subset of a maximal $\omega$-limit set by Theorem \ref{prop:max} and it contains points from $\omega_f(y)\subset D(X)$, we have also $\alpha(\{\tilde{q}_j\}_{j\leq 0})\subset D(X)$.

It remains to show that $\{\tilde{q}_j\}_{j\leq 0}$ with $\tilde{q}_0=x$ exists for every $x \in X\setminus \mathcal{I}(X)$. But the sets $\phi^{-1}(\tx_j)$ form an inverse sequence by assumption, that is $f(\phi^{-1}(\tx_{j-1}))=\phi^{-1}(\tx_j)$, for every $j\leq 0$, therefore we can construct $\{\tilde{q}_j\}_{j\leq 0}$ for every $\tilde{q}_0\in \phi^{-1}(x)$, where $x$ is an arbitrary point from an open connected set $U\subset Y$, hence for every $\tilde{q}_0\in \phi^{-1}(U)$. Denote $V=\phi^{-1}(U)$. The result follows by Lemma \ref{lem:mixingE}.

 If $Y$ has $n$ components then $g^n$ is mixing on each of the periodic components $Y_i$, $i=0,\ldots,n-1$, and we can decompose $\omega_g(\phi(y))$ into infinite sets $\omega_g(\phi(y))=\bigcup_{i=0}^{n-1}g^i(\omega_{g^n}(\phi(y)))$. Without loss of generality assume $\omega_{g^n}(\phi(y))\subset Y_0$. The same construction as above gives us the backward brach $\{\tx_j\}_{j\leq 0}$ such that $\tx_0=x$ and $\alpha_{g^n}(\{\tx_j\}_{j\leq 0}) = \omega_{g^n}(\phi(y))$, for every $x$ in an open set $U\subset Y_0$. Set:
$$z_i = \begin{cases}\tx_j \hfill\text{ if }i=j\cdot n,\text{ for every } j\in\mathbb{N}_0,\\ g^k(\tx_j)\hspace{0.5cm}\text{ if }i=j\cdot n+k,\text{ for every } 0<k<n,  j\in\mathbb{N}_0.\end{cases}$$
By continuity of $g$, $\alpha_g(\{z_i\}_{i\leq 0}) = \omega_{g}(\phi(y))$. We finish the proof in the same manner as above.
\end{proof}
Theorem \ref{thm:alphaPosEntMap} can be stated in various forms. We describe them in the following series of facts and remarks.
\begin{fact}\label{rem:alpha} The set $\alpha_f(\{z_j\}_{j\leq 0})\setminus \omega_f(y)$ is at most countable and consists from isolated points of $\alpha_f(\{z_j\}_{j\leq 0})$. Consequently, if every isolated point $x$ of $\alpha_f(\{z_j\}_{j\leq 0})$ has $[x]_{\sim}=\{x\}$ then $\omega_f(y)=\alpha_f(\{z_j\}_{j\leq 0}$.
\end{fact}
\begin{proof}
Since $[x]_{\sim}=\{x\}$, for all but countably many $x\in D(X)$, and $[x]_{\sim}\cap D(X)=End([x]_{\sim})$ is a finite set, for every $x\in X$, we have $[\omega_f(y)]_{\sim}\cap D(X)\setminus \omega_f(y)$ countable. Moreover, points from $[\omega_f(y)]_{\sim}\cap D(X)\setminus \omega_f(y)$ are isolated in $[\omega_f(y)]_{\sim}\cap D(X)$. If $\{x_i\}\to \{x\}$ is a converging sequence such that $\{x_i\}_{i>0}\subset [\omega_f(y)]_{\sim}\cap D(X)$, then $x_i\in[y_i]_{\sim}$ where $y_i\in \omega_f(y)$, for every $i>0$. If the sequence is not eventually constant, we can assume $[y_i]_{\sim}$ are pairwise disjoint subgraphs of $X$ with diameter tending to 0 as $i\to \infty$ (we can pass to a subsequence if necessary since there is at most finitely many indeces $k$ such that $[y_k]_{\sim}=[y_i]_{\sim}$, for any $i>0$). Since $\omega_f(y)$ is closed, we have $y_i\to x$ and $x\in \omega_f(y)$. Therefore points from  $([\omega_f(y)]_{\sim}\cap D(X))\setminus \omega_f(y)$ are never accumulation points of $[\omega_f(y)]_{\sim}\cap D(X)$. The same holds for $\alpha_f(\{z_j\}_{j\leq 0})\setminus \omega_f(y)$ since $\alpha_f(\{z_j\}_{j\leq 0})\setminus\omega_f(y)\subseteq([\omega_f(y)]_{\sim}\cap D(X))\setminus\omega_f(y)$.
\end{proof}

\begin{rem}\label{rem:omit}
If we omit the condition $f([x]_{\sim})= [f(x)]_{\sim}$, for every $x\in X$, in the assumption of Theorem \ref{thm:alphaPosEntMap}, then we obtain the following weaker result:\\
Let $D(X)$ be a basic set. Then, for every $x \in X\setminus \mathcal{I}(X)$ and every $\omega$-limit set $\omega_f(y)$ such that $\omega_f(y)\subset D(X)$ is infinite, there exists a backward branch $\{z_j\}_{j\leq 0}$ such that $z_0\in[x]_{\sim}$ and $\omega_f(y)\subseteq \alpha_f(\{z_j\}_{j\leq 0})\subseteq[\omega_f(y)]_{\sim}\cap D(X)$. 
\end{rem}
\begin{rem}
The inaccesible points from $\mathcal{I}(X)$ have only finite $\alpha$-limit sets of backward branches being a subset of $D(X)$. Nevertheless, they may have many infinite $\alpha$-limit sets of backward branches being a subset of other basic set $B(Z)$ such that $D(X)\cap B(Z)\neq \emptyset$.
\end{rem}

Combining together Fact \ref{rem:alpha}, Remark \ref{rem:omit} and the fact that $[x]_{\sim}=\{x\}$ for all but countably many points $x$ from $D(X)$ we get the following corollary.
\begin{cor}\label{cor:posEnt}
Let $D(X)$ be a basic set. Then for all but countably many points $x\in D(X)$ and every $\omega$-limit set $\omega_f(y)$ such that $\omega_f(y)\subset D(X)$ is infinite, there exists a backward branch $\{z_j\}_{j\leq 0}$ such that $z_0=x$ and $\alpha_f(\{z_j\}_{j\leq 0})=\omega_f(y)\cup R$ where $R$ is at most countable subset of isolated points of $\alpha_f(\{z_j\}_{j\leq 0})$. Moreover, if every isolated point $x$ of $\alpha_f(\{z_j\}_{j\leq 0})$ has $[x]_{\sim}=\{x\}$ then $R$ is empty.
\end{cor}
We leave the next question open for futher resaerch.
\begin{que}
Let $D(X)$ be a basic set. Is it true that for every $x\in D(X)\setminus \mathcal{I}(X)$ and every $\omega$-limit set $\omega_f(y)$ such that $\omega_f(y)\subset D(X)$ and $\Orb(y)$ is infinite, there exists a backward branch $\{z_j\}_{j\leq 0}$ such that $z_0=x$ and $\alpha_f(\{z_j\}_{j\leq 0})=\omega_f(y)$?
\end{que}

The following example shows that Theorem \ref{thm:alphaPosEntMap} can not be applied when $\omega_f(y)$ is finite. In particular, we will show that there is a basic set $D(X)$ and a fixed point $p\in D(X)$ such that there is no backward branch $\{z_j\}_{j\leq 0}$ with $z_0\in D(X)$ and $\{p\}\subset\alpha_f(\{z_j\}_{j\leq 0})\subset [p]_{\sim}$ (with the exception of the constant backward branch $z_j=p$, for every ${j\leq 0}$).
\begin{exmp}
Let $g$ be a mixing map of the unit interval $I$ and $p,q,r$ be points from Figure \ref{fig:periodicOrb}. Let $D(I,f)$ be a basic set such that there is an almost conjugacy $\phi$ between $f|_{D(I,f)}$ and $g$ with the following properties: $\phi^{-1}(x)$ is not a singleton, for every $x\in\cup_{i\geq 0}f^{-i}(p)$, $\phi^{-1}(p)$ is an $f$-invariant interval $[p_1,p_2]$, where $p_1,p_2$ are fixed points with respect to $f$, $\phi^{-1}(q)$ is an interval $[q_1,q_2]$ such that $f(q_1)=p_1$ and $f(q_2)=p_2$ and $\phi^{-1}(r)$ is an interval $[r_1,r_2]$ such that $f(r_1)=p_2$ and $f(r_2)=p_1$. The fixed point $p_1$ can be reached only by a backward branch from the invariant interval $[p_1,p_2]$, since every backward branch $\{z_j\}_{j\leq 0}$ converging to $p_1$ from the left side or from both sides has $\alpha_f(\{z_j\}_{j\leq 0})$ containing $q_1$ or $r_2$.\end{exmp}

\begin{figure}[h!]\label{fig:periodicOrb}

\begin{tikzpicture}[scale=0.5,line cap=round,line join=round,>=triangle 45,x=1cm,y=1cm]
\clip(-0.5,-0.5) rectangle (10.5,10.5);
\draw [line width=1pt] (0,10)-- (0,0);
\draw [line width=1pt] (0,0)-- (10,0);
\draw [line width=1pt] (10,0)-- (10,10);
\draw [line width=1pt] (0,10)-- (10,10);
\draw [line width=1pt] (0,0)-- (10,10);
\draw [line width=1pt,dash pattern=on 1pt off 2pt] (5,5)-- (5,0);
\draw [line width=1pt,dash pattern=on 1pt off 2pt] (1.25,5)-- (1.25,0);
\draw [line width=1pt,dash pattern=on 1pt off 2pt] (8.75,5)-- (8.75,0);
\draw [line width=1pt] (2.505719867829907,10)-- (0,0);
\draw [line width=1pt] (2.505719867829907,10)-- (5,5);
\draw [line width=1pt] (7.490873240958322,10)-- (5,5);
\draw [line width=1pt] (7.490873240958322,10)-- (10,0);
\begin{scriptsize}
\draw [fill=blue] (5,0) circle (2.5pt);
\draw[color=blue] (5.2,0.4) node {p};
\draw [fill=blue] (8.75,0) circle (2.5pt);
\draw[color=blue] (9,0.4) node {r};
\draw [fill=blue] (1.25,0) circle (2.5pt);
\draw[color=blue] (1.5,0.4) node {q};
\end{scriptsize}

\end{tikzpicture}
\caption{A map mixing interval map $g$ with a fixed point $p$ such that every backward branch $\{z_j\}_{j\leq 0}$ with $\{p\}= \alpha_g(\{z_j\}_{j\leq 0})$ converges to $p$ only from the right side. Every backward branch $\{z_j\}_{j\leq 0}$ converging to $p$ from the left side or from both sides has $\alpha_g(\{z_j\}_{j\leq 0})\cap \{q,r\}\neq\emptyset$, where $q,r$ are preimages of $p$.}\label{fig:zeroentropy}
\end{figure}
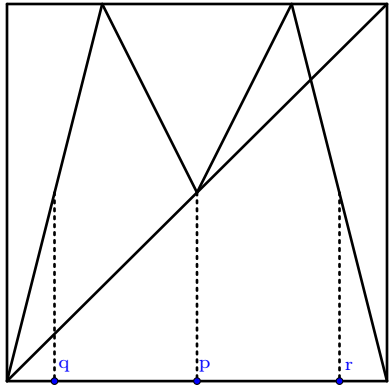

\begin{thm}\label{thm:omegaPosEntMap}
 Let $D(X)$ be a basic set. Then, for every backward branch $\{x_j\}_{j\leq 0}\subset X$ such that $\alpha_f(\{x_j\}_{j\leq 0})\subset D(X)$ there is a point $y\in X$ such that $\omega_f(y)\subseteq [\alpha_f(\{x_j\}_{j\leq 0})]_{\sim}\cap D(X)$. Moreover, if the set $\{x\in\alpha_f(\{x_j\}_{j\leq 0}): [x]_{\sim}=\{x\}\}$ is dense in $\alpha_f(\{x_j\}_{j\leq 0})$ then $\alpha_f(\{x_j\}_{j\leq 0})\subseteq \omega_f(y)\subseteq [\alpha_f(\{x_j\}_{j\leq 0})]_{\sim}\cap D(X)$.
\end{thm}
\begin{proof}
If $\alpha_f(\{x_j\}_{j\leq 0})$ is finite then by Lemma \ref{lem:alphaPer} it is a periodic orbit of some point $y \in X$ and $\omega_f(y) =\alpha_f(\{x_j\})_{j\leq 0}$. Assume that $f|_{D(X)}$ is almost conjugate to a mixing graph $g\colon Y\rightarrow Y$, i.e. the cycle of graphs $Y$ has only one component. Let $\{x_j\}_{j\leq 0}\subset X$ be such that $\alpha_f(\{x_j\}_{j\leq 0})\subset D(X)$ is infinite. Note that $\phi(\alpha_f(\{x_j\}_{j\leq 0})=\alpha_g(\{\phi(x_j)\}_{j\leq 0})$. \\
First assume that the set $S:=\{x\in\alpha_f(\{x_j\}_{j\leq 0}): [x]_{\sim}=\{x\}\}$ is dense in $\alpha_f(\{x_j\}_{j\leq 0})$. Let $\delta>0$ and $P(\delta),\gamma$ be from Lemma \ref{lem:P}. We can find a finite set $M\subset S$ such that $\bigcup_{x\in M}B(x,\delta)\supset \alpha_f(\{x_j\}_{j\leq 0})$ and obviously $\phi(M)\cap P(\delta)=\emptyset$. Applying the construction from the proof of Theorem \ref{thm:mixingOmega} to $\alpha_g(\{\phi(x_j)\}_{j\leq 0})\subset Y$ we obtain a sequence of periodic orbits $\{\Orb(p_k)\}_{k>0}\subset Y$ with increasing periods $\{d_k\}_{k>0}$ such that $d_{H}(\Orb(p_k),\alpha_g(\{\phi(x_j)\}_{j\leq 0})<2\epsilon_k$, for every $k> 0$. Since $\epsilon_k$ goes to 0 as $k\to \infty$, we can find $k>0$ such that $2\epsilon_k<\min\{\gamma,\dist(\phi(M),P(\delta))\}$. Then for every $x\in M$ there is $i\in\{0\ldots,d_k-1\}$ with $g^i(p_k)\in B(\phi(x),2\epsilon_k)$. By Lemma \ref{lem:P}, $\diam\phi^{-1}(B(\phi(x),2\epsilon_k))<\delta$, for every $x\in M$. Therefore $\phi^{-1}(B(\phi(x),2\epsilon_k))\subset B(x,\delta)$, for every $x\in M$. Since $p_k$ has period $d_k$, we can find an $f$-periodic point $q_k\in\phi^{-1}(p_k)$ with period at least $d_k$ and obviously $f^i(q_k)\in\phi^{-1}(g^i(p_k))$, for every $i\in\{0\ldots,d_k-1\}$. It follows that for every $x\in M$ there is $i\in\{0\ldots,d_k-1\}$ with $f^i(q_k)\in B(x,\delta)$. Since  we have assumed that $\bigcup_{x\in M}B(x,\delta)$ covers $\alpha_f(\{x_j\}_{j\leq 0})$, we conclude $\cup_{z\in\Orb(q_k)}B(z,2\delta)\supset \alpha_f(\{x_j\}_{j\leq 0})$ and $\alpha_f(\{x_j\}_{j\leq 0})\subset  \omega_f(y)$ where $ \omega_f(y)$ is the Hausdorff limit of the sequence $\{\Orb(q_k)\}_{k>0}$. 

On the other hand, we can use the sequence of periodic orbits $\{\Orb(p_k)\}_{k>0}\subset Y$ from the proof of Theorem \ref{thm:mixingOmega} to construct $\omega_f(y)$ as the Hausdorff limit of the sequence $\{\Orb(q_k)\}_{k> 0}$ with $\Orb(q_k)\subset \phi^{-1}(\Orb(p_k))$, for every $k> 0$, regardless of the existence of the set $S$. Then $\phi(\omega_f(y))=\alpha_g(\{\phi(x_j)\}_{j\leq 0})=\phi(\alpha_f(\{x_j\}_{j\leq 0})$ and at the same time $\omega_f(y)\subseteq D(X)$ which gives $\omega_f(y)\subseteq [\alpha_f(\{x_j\}_{j\leq 0})]_{\sim}\cap D(X).$ 

 If $Y$ has $n$ components $Y_0,\ldots,Y_{n-1}$ then $g^n$ is mixing on each of the periodic components $Y_i$ and $X=\cup_{i=0}^{n-1}\phi^{-1}(Y_i)$ where sets $\phi^{-1}(Y_i)$ are connected, pairwise disjoint with possibly non-empty intersection in the endpoints and they form a cycle of period $n$. We can decompose $\alpha_f(\{x_j\}_{j\leq 0})$ into infinite sets $\alpha_f(\{x_j\}_{j\leq 0})=\bigcup_{i=0}^{n-1}f^i(\alpha_{f^n}(\{x_{n\cdot j}\}_{j\geq 0})$. Without loss of generality assume $\alpha_{f^n}(\{x_{n\cdot j}\}_{j\geq 0})\subset \phi^{-1}(Y_0)$. First we show that if the set $S=\{x\in\alpha_f(\{x_j\}_{j\leq 0}): [x]_{\sim}=\{x\}\}$ is dense in $\alpha_f(\{x_j\}_{j\leq 0})$ then $S\cap \phi^{-1}(Y_0)$ is dense in $\alpha_{f^n}(\{x_{n\cdot j}\}_{j\geq 0})$. Since $S\cap \Int(\phi^{-1}(Y_0))$ is dense in the open (with respect to the subspace topology) set $\alpha_{f^n}(\{x_{n\cdot j}\}_{j\geq 0})\cap \Int(\phi^{-1}(Y_0))$ it sufficies to show that $\alpha_{f^n}(\{x_{n\cdot j}\}_{j\geq 0})\cap End(\phi^{-1}(Y_0))$ is not isolated from $S\cap \phi^{-1}(Y_0)$. If $z\in \alpha_{f^n}(\{x_{n\cdot j}\}_{j\geq 0})\cap End(\phi^{-1}(Y_0))$ has a pre-image in $\alpha_{f^n}(\{x_{n\cdot j}\}_{j\geq 0})\cap \Int(\phi^{-1}(Y_0))$ then $z$ is either a limit point of $S$ or belongs to $S$ by continuity of $f^n$ and invariance of $S$. If $z$ has only pre-images in $\alpha_{f^n}(\{x_{n\cdot j}\}_{j\geq 0})\cap End(\phi^{-1}(Y_0))$ then $z$ is a periodic point and consequently $z$ being isolated from $S\cap \phi^{-1}(Y_0)$ implies that $z$ is isolated from $S$. But this is impossible since $z\in\alpha_f(\{x_j\}_{j\leq 0})$.

Now we can apply the above procedure to the $\alpha$-limit set $\alpha_{f^n}(\{x_{n\cdot j}\}_{j\geq 0})$ and obtain $ \omega_{f^n}(y)$ with $\alpha_{f^n}(\{x_{n\cdot j}\}_{j\geq 0})\subseteq \omega_{f^n}(y)\subseteq [\alpha_{f^n}(\{x_{n\cdot j}\}_{j\geq 0})]_{\sim}\cap D(X)$ (in case $S$ is not dense in $\alpha_f(\{x_j\}_{j\leq 0})$ we consider only the second inclusion). Obviously $f^i(\alpha_{f^n}(\{x_{n\cdot j}\}_{j\geq 0}))\subseteq f^i(\omega_{f^n}(y))\subseteq f^i( [\alpha_{f^n}(\{x_{n\cdot j}\}_{j\geq 0})]_{\sim})\cap D(X)$, for $i=0,\ldots,n-1$, and therefore:
\begin{multline}
\alpha_f(\{x_j\}_{j\leq 0})=\bigcup_{i=0}^{n-1}f^i(\alpha_{f^n}(\{x_{n\cdot j}\}_{j\geq 0}))\subseteq \omega_f(y)=\bigcup_{i=0}^{n-1}f^i(\omega_{f^n}(y))\\ \subseteq [\alpha_f(\{x_j\}_{j\leq 0})]_{\sim}\cap D(X)=\bigcup_{i=0}^{n-1}f^i([\alpha_{f^n}(\{x_{n\cdot j}\}_{j\geq 0})]_{\sim})\cap D(X).
\end{multline}

\end{proof}
Repeating the arguments from Fact \ref{rem:alpha} we can show that $\omega_f(y)\setminus \alpha_f(\{x_j\}_{j\leq 0})$ consists of isolated points of $\omega_f(y)$ and the following corollary holds.
\begin{cor}\label{cor:omegaPosEntMap}
 Let $D(X)$ be a basic set and $\{x_j\}_{j\leq 0}\subset X$ be a backward branch such that $\alpha(\{x_j\}_{j\leq 0})\subset D(X)$.  If the set $\{x\in\alpha_f(\{x_j\}_{j\leq 0}): [x]_{\sim}=\{x\}\}$ is dense in $\alpha_f(\{x_j\}_{j\leq 0})$ then there is a point $y\in X$ such that $\omega_f(y)=\alpha_f(\{x_j\}_{j\leq 0})\cup R$ where $R$ is at most countable subset of isolated points of $\omega_f(y)$. Moreover, if every isolated point $x$ of $\omega_f(y)$ has $[x]_{\sim}=\{x\}$ then $R$ is empty.
 
\end{cor}

In the previous Section \ref{sec:zeroEnt} we have proved that if $h(f)=0$ then every $\alpha_f(\{x_j\}_{j\leq 0})$ is a minimal set, hence it is an $\omega$-limit set of any point from $\alpha_f(\{x_j\}_{j\leq 0})$. Clearly, the same holds for $f$ with positive entropy and $\alpha_f(\{x_j\}_{j\leq 0})$ being a subset of one of the three maximal $\omega$-limit sets which are in common to both zero entropy graph maps and positive entropy graph maps - solenoidal sets, circumferential sets and periodic orbits. In the light of Theorem \ref{prop:max} we can conclude that, for any graph map $f$, every $\alpha$-limit set of a backward branch is an $\omega$-limit set, providing the answer to the following question turns out positive.

\begin{que}
Let $D(X)$ be a basic set and $\alpha_f(\{x_j\}_{j\leq 0})\subset D(X)$, for a backward branch $\{x_j\}_{j\leq 0}$. Is $\alpha_f(\{x_j\}_{j\leq 0})=\omega_f(y)$, for some $y\in X$?
\end{que}
\section*{Acknowledgements}

M. Fory\'s-Krawiec was supported in part by the National Science Centre, Poland (NCN), grant SONATA BIS no. 2019/34/E/ST1/00237: "Topological and Dynamical Properties in Parameterized Families of Non-Hyperbolic Attractors: the inverse limit approach". P. Oprocha was supported in part by Polish Ministry of Science and Higher Education, grant no. 477132/PnH2/2020.

\begin{table}[h]

\begin{tabular}[t]{b{1.5cm} m{10.5cm}}

\includegraphics [width=.09\textwidth]{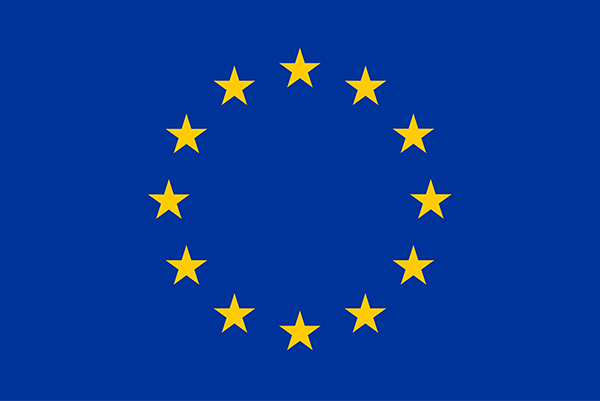} & 
This research is part of a project that has received funding from the European Union's Horizon 2020 research and innovation programme under the Marie Sk\l odowska-Curie grant agreement No 883748.

\end{tabular}
\end{table}

\end{document}